\documentclass[reqno,12pt]{amsart}%
\usepackage[body={17cm,22cm}]{geometry}

\usepackage[T1]{fontenc}
\usepackage{graphicx}
\usepackage{mathrsfs}
\usepackage[french]{babel}
\usepackage{amscd,latexsym,amsthm,amsfonts,amssymb,amsmath,amsxtra}
\usepackage[colorlinks, urlcolor=blue,  citecolor=blue]{hyperref}
\usepackage[all]{xy}%
\providecommand{\U}[1]{\protect\rule{.1in}{.1in}}
\RequirePackage{amsmath}
\RequirePackage{amssymb}

\newcommand{\BC}{{\mathbb {C}}}

\newcommand{\BG}{{\mathbb {G}}}

\newcommand{\BQ}{{\mathbb {Q}}}

\newcommand{\BZ}{{\mathbb {Z}}}

\newcommand{\CF}{{\mathcal {F}}}

\newcommand{\CH}{{\mathcal {H}}}

\newcommand{\CT}{{\mathcal {T}}}

\newcommand{\RA}{{\mathbf {A}}}

\newcommand{\Br}{{\mathrm{Br}}}

\newcommand{\Coker}{{\mathrm{Coker}}}
\newcommand{\coker}{{\mathrm{coker}}}

\newcommand{\Gal}{{\mathrm{Gal}}}

\newcommand{\Hom}{{\mathrm{Hom}}}

\renewcommand{\Im}{{\mathrm{Im}}}

\newcommand{\Ker}{{\mathrm{Ker}}}

\newcommand{\Pic}{\mathrm{Pic}}

\newcommand{\Spec}{{\mathrm{Spec}}}

\font\cyr=wncyr10
\newcommand{\Sha}{\hbox{\cyr X}}

\newcommand{\iso}{\stackrel{\sim}{\rightarrow} }

\newcommand{\sbt}{\subset}

\newcommand{\bk}{\bar{k}}

\newcommand{\et}{{\rm{\acute et}}}

\numberwithin{equation}{section}

\theoremstyle{remark}

\newtheorem{defi}{\rm{\textbf{D\'efinition}}}[section]

\theoremstyle{plain}

\newtheorem{thm}[defi]{\rm{\textbf{Th\'eor\`eme}}}

\newtheorem{cor}[defi]{\rm{\textbf{\textbf{Corollaire}}}}

\newtheorem{lem}[defi]{\rm{\textbf{Lemme}}}

\newtheorem{prop}[defi]{\rm{\textbf{\textbf{Proposition}}}}

\begin{document}

\title[Sous-groupe invariant]
{Sous-groupe de Brauer invariant pour un groupe alg\'ebrique connexe quelconque}

\author{Yang CAO}

\address{Yang CAO \newline 	University of Science and Technology of China,
\newline 96 Jinzhai Road,
 \newline	230026 Hefei, Chine}

\email{yangcao1988@ustc.edu.cn; yangcao1988@gmail.com}

\date{\today.}

\maketitle

\begin{abstract}

 Dans cet article, pour une vari\'et\'e lisse X munie d'une action d'un groupe alg\'ebrique connexe G (non n\'ecessairement lin\'eaire), on introduit la notion de sous-groupe de Brauer invariant et la notion d'obstruction de Brauer-Manin \'etale invariante. Ensuite, on montre que cette obstruction \'equivaut \`a l'obstruction de Brauer-Manin \'etale.
 Ceci g\'en\'eralise la notion principale et le r\'esultat cl\'e de l'aticle pr\'ec\'edent de l'auteur et ceci g\'en\'eralise aussi un r\'esultat de B. Creutz. 
 
Mots cl\'es: Groupe alg\'ebrique, groupe de Brauer, principle de Hasse.

\medskip

\indent
Summary. 
In this paper, for a smooth variety X equipped with an action of a connected algebraic group G (not necessary linear),
  we introduce the notion of invariant Brauer sub-group and the  notion of invariant \'etale Brauer-Manin obstruction. 
  Then we prove that this obstruction is equivalent to the \'etale Brauer-Manin obstruction.
This extends the main notion and the key result of author's previous article and this also extends a result of B. Creutz.

\end{abstract}

\tableofcontents

\section{Introduction}

Soit $k$ un corps de nombres. 
 On note $\Omega_k $ l'ensemble des places du corps de nombres $ k $ et $ \RA_k $ l'anneau des ad\`eles de $k$.
Pour chaque $ v \in \Omega_k $, on note $k_v$ le compl\'et\'e de $k$ en $v$.  
Pour une $k$-vari\'et\'e $X$, on note $X(\RA_k)$ l'ensemble des points ad\'eliques de $X$ (voir \cite{Co}).

Soit $X$ une $k$-vari\'et\'e alg\'ebrique et $\Br (X)$ son groupe de Brauer cohomologique.
Pour $B$ sous-ensemble de $\Br (X)$, on d\'efinit
$$ X ({\RA}_k)^ B = \{(x_v)_{v \in \Omega_k} \in X ({\RA}_k): \ \ \sum_{v \in \Omega_k} \ inv_v (\xi (x_v)) = 0\in \BQ/\BZ, \ \ \forall \xi \in B \}.$$ 
Comme l'a remarqu\'e Manin, la th\'eorie du corps de classes donne $ X (k) \subseteq X (\RA_k)^B $.
Ceci d\'efinit une obstruction au principe de Hasse pour $X$, appel\'ee \emph{obstruction de Brauer-Manin}.

 Depuis 1970s, divers auteurs (Manin, Colliot-Th\'el\`ene, Sansuc, Skorobogatov, Harari, Demarche, Poonen, Xu et l'auteur)  
  ont d\'ecrit diverses obstructions au principe de Hasse
et montr\'e que l'ensemble de Brauer-Manin \'etale $X(\RA_k)^{\et,\Br}$ (voir (\ref{Bthm2e})) est le plus fin (voir \cite{Ha02,D09,Sk1,CDX,C5}).
En particulier,  l'auteur  a montr\'e l'\'equivalence entre l'obstruction de descente et l'obstruction de descente it\'er\'ee (\cite[Thm. 1.2]{C5}, une question ouverte de Poonen).
Deux \'etapes cl\'es de  la d\'emonstration de \cite[Thm. 1.2]{C5} sont de d\'efinir l'ensemble de Brauer-Manin \'etale invariant $X(\RA_k)^{G-\et,\Br_G}$ (\cite[(1.2)]{C5}) 
et de montrer (\cite[Thm. 1.4]{C5}):
\begin{equation}\label{Be1}
X(\RA_k)^{G-\et,\Br_G}=X(\RA_k)^{\et,\Br},
\end{equation}
o\`u $G$ est un groupe lin\'eaire connexe et $X$ est une $G$-vari\'et\'e lisse. 

Dans cette article, on g\'en\'eralise cette notion dans le cas o\`u $G$ est un groupe alg\'ebrique connexe quelconque (pas forcement lin\'eaire): 
on d\'efinit $X(\RA_k)^{G-\et,\Br'_G}$ dans ce cas et 
le r\'esultat principal de cet article (th\'eor\`eme \ref{Thm1}) est une  g\'en\'eralisation de (\ref{Be1}).
De plus, le th\'eor\`eme \ref{Thm1} g\'en\'eralise un r\'esultat de B. Creutz \cite[Thm. 2]{Cr} sur les torseurs des vari\'et\'es ab\'eliennes (voir le corollaire \ref{cor1mainthm}).

\bigskip

Donnons maintenant des \'enonc\'es pr\'ecis.

\begin{defi}\label{BDef1}
Soient $G$ un $k$-groupe alg\'ebrique connexe et $(X,\rho )$ une $G$-vari\'et\'e lisse connexe. 

(1) \emph{Le sous-groupe cohomologique invariant de degr\'e $i$} de $X$ est le sous-groupe: 
$$H^i_G(X,\mu_{\infty}):=\{b\in H^i(X,\mu_{\infty})\ :\ (\rho^*(b)-p_2^*(b))\in p_1^*H^i(G,\mu_{\infty})\},$$
o\`u $i\in \BZ_{\geq 0}$, $G\times X\xrightarrow{p_1}G$, $G\times X\xrightarrow{p_2}X$ sont les projections,
et $G\times X\xrightarrow{\rho}X$ est l'action de $G$.

(2) \emph{Le sous-groupe de Brauer $G$-invariant} de $X$ est le sous-groupe
$$\Br_G(X):= \{a\in \Br(X)\ :\ (\rho^*(a)-p_2^*(a))\in p_1^*\Br(G)\} $$
et \emph{le sous-groupe de Brauer $G$-invariant am\'elior\'e} est le sous groupe
$$\Br'_G(X):=\Im\left(H^2_G(X,\mu_{\infty})\sbt H^2(X,\mu_{\infty})\to \Br(X)\right)\sbt \Br(X).$$ 
\end{defi}

Ce sous-groupe v\'erifie des propri\'et\'es esp\'er\'ees. Par d\'efinition, $\Br'_G(X)\sbt \Br_G(X)$.
Dans le cas o\`u $X=G$, on a $\Br_{2/3}(G)=\Br'_G(G)$ (proposition \ref{prop21.2} (2)) et
dans le cas o\`u $X$ est g\'eom\'etriquement int\`egre, on a: $\Br_{2/3}(X)\sbt \Br'_G(X)$ (proposition \ref{prop21.2} (1)), 
 o\`u $\Br_{2/3}(X)$ est d\'efini dans la d\'efinition \ref{DefiBr2/3}.
 Si $G$ est lin\'eaire, on a $\Br'_G(X)=\Br_G(X)$ (corollaire \ref{corlinBr1et2/3}), mais il existe une vari\'et\'e ab\'elienne $A$ telle que $\Br'_A(A)\neq \Br_A(A)$ (cf. \cite[Thm. 1.2, 1.3]{OSVZ}). 
 Dans le cas g\'en\'eral,  $H^2_G(X,\mu_{\infty})$ s'ins\`ere dans des suites exactes et ceci implique des propri\'et\'es importantes de $\Br'_G(X)$ (voir \S \ref{section4}, \S \ref{5}).
Pour $\Br_G(X)$, l'auteur conjecture qu'il n'existe pas de propri\'et\'e comme les corollaires \ref{corbraueralgebraic} et \ref{propbrauersuj}.

Dans cet article, tout torseur est un torseur \`a droite.
Soient $F$ un $k$-groupe alg\'ebrique et  $f: Y\to X$ un $F$-torseur. 
Pour tout 1-cocycle $\sigma\in Z^1(k,F)$, on note $F_{\sigma}$, respectivement  $f_{\sigma}: Y_{\sigma}\to X$ le   tordu du $k$-groupe 
$F$, respectivement du torseur $f$,
 par le 1-cocycle $\sigma$.
Alors $f_{\sigma}$ est un $F_{\sigma}$-torseur. La classe d'isomorphisme du $k$-groupe $F_{\sigma}$, respectivement du torseur $f_{\sigma}$, ne d\'epend que de la classe de $\sigma$ dans $H^1(k,F)$.
Par abus de notation, \'etant donn\'ee une classe $[\sigma] \in H^1(k,F)$, on
note $F_{\sigma}=F_{[\sigma]}$ et $f_{\sigma}=f_{[\sigma]}$.

Pour une vari\'et\'e lisse $X$, Skorobogatov (\cite{Sk99}) et Poonen d\'efinissent (\cite[\S 3.3]{P}) l'ensemble suivant
\begin{equation}\label{Bthm2e}
X(\RA_k)^{\et, \Br}:=\bigcap_{\stackrel{f: Y\xrightarrow{F}X,}{ F\ \text{fini}}} \bigcup_{\sigma\in H^1(k,F)}f_{\sigma}(Y_{\sigma}(\RA_k)^{\Br(Y_{\sigma})}) ,
\end{equation}
et on a une inclusion $X(k)\sbt X(\RA_k)^{\et, \Br}$.
Ceci d\'efinit une obstruction au principe de Hasse pour $X$, appel\'ee \emph{obstruction de Brauer-Manin \'etale}.

\begin{defi}
Soient $G$ un $k$-groupe alg\'ebrique connexe et  $X$ une $G$-vari\'et\'e lisse.

(1) \emph{Le sous-groupe de Brauer $G$-invariant am\'elior\'e (resp. le sous-groupe de Brauer $G$-invariant )} de $X$ est le sous-groupe 
$\Br'_G(X)\sbt \Br(X)$ (resp. $\Br_G(X)\sbt \Br(X)$) des \'el\'ements $\alpha$
 v\'erifiant $\alpha|_{X'}\in \Br'_G(X')$ (resp. $\alpha|_{X'}\in \Br_G(X')$) pour toute composante connexe $X'$ de $X$.

(2) Soit $F$ un $k$-groupe fini.  
Un $F$-torseur $Y\xrightarrow{f}X$ est \emph{$G$-compatible} s'il existe une action de $G$ sur $Y$ telle que $f$ soit un $G$-morphisme.
\end{defi}

D'apr\`es \cite[Prop. 3.3]{C5}, l'action de $G$ sur $Y$ v\'erifiant les conditions ci-dessus est unique et
  le $F_{\sigma}$-torseur $f_{\sigma}$ est aussi $G$-compatible pour tout $\sigma\in H^1(k,F)$. 
On d\'efinit les variantes de $X(\RA_k)^{\et,\Br}$ suivantes:
\begin{equation}
X(\RA_k)^{G-\et, \Br_G}:=\bigcap_{\stackrel{f: Y\xrightarrow{F}X\ G-\text{compatible} ,}{ F\ \text{fini}}} \bigcup_{\sigma\in H^1(k,F)}f_{\sigma}(Y_{\sigma}(\RA_k)^{\Br_G(Y_{\sigma})}) .
\end{equation}
et
\begin{equation}\label{def1e1}
X(\RA_k)^{G-\et, \Br'_G}:=\bigcap_{\stackrel{f: Y\xrightarrow{F}X\ G-\text{compatible} ,}{ F\ \text{fini}}} \bigcup_{\sigma\in H^1(k,F)}f_{\sigma}(Y_{\sigma}(\RA_k)^{\Br'_G(Y_{\sigma})}) .
\end{equation}
Alors $X(k)\sbt X(\RA_k)^{\et,\Br}\sbt X(\RA_k)^{G-\et, \Br_G}\sbt X(\RA_k)^{G-\et, \Br'_G}$.
Ceci d\'efinit des obstructions au principe de Hasse pour $X$, appel\'ee \emph{obstruction de Brauer-Manin \'etale invariante}.

Le th\'eor\`eme suivant g\'en\'eralise \cite[Thm. 1.4]{C5}.

\begin{thm}\label{Thm1}
Soient $G$ un groupe alg\'ebrique connexe et $X$ une $G$-vari\'et\'e lisse. 
Alors 
$$X(\RA_k)^{\et,\Br}= X(\RA_k)^{G-\et, \Br_G}= X(\RA_k)^{G-\et, \Br'_G}.$$
\end{thm}

Dans le cas o\`u $X$ est un $G$-espace homog\`ene \`a stabilisateur g\'eom\'etrique connexe, le th\'eor\`eme \ref{Thm1} implique $X(\RA_k)^{\Br'_G(X)}=X(\RA_k)^{\et, \Br} $ (corollaire \ref{cormainthm}). 
Si $X=G/H$ avec $H\subset G$ un sous-groupe connexe, dans  \cite[Thm. 1.4]{BD}, Borovoi et Demarche \'etablissent l'approximation forte par rapport \`a $\Br_1(X,G)$ hors des places archim\'ediens pour $X$  (avec des conditions classiques), 
o\`u $\Br_1(X,G):=\ker(\Br(X)\to \Br(G_{\bk}))$.
En fait, d'apr\`es (\ref{sansucthm-cor1-e1}), on a $\Br'_G(X)\subset  \Br_1(X,G)$, mais ces deux groupes ne sont pas \'egaux en g\'en\'eral:
par exemple, si $X=G$ et $G$ est une vari\'et\'e ab\'elienne v\'erifiant $H^1(k,\mathrm{NS}(G_{\bk}))\neq 0$,
 on a $\Br'_G(G)=\Br_{2/3}(G)$, $\Br_1(X,G)=\Br_1(G)$ et $\Br_{2/3}(G)\neq \Br_1(G)$, o\`u $\mathrm{NS}(G_{\bk})$ est le groupe de N\'eron-Severi g\'eom\'etrique de $G$.
Donc notre r\'esultat affirme que $\Br'_G(X)$ est suffit pour \'etablir l'approximation forte, qui renforce le r\'esultat de Borovoi et Demarche.

Un cas sp\'ecial du corollaire \ref{cormainthm} est (en utilisant la proposition \ref{prop21.2} (2)):

\begin{cor}\label{cor1mainthm}
Soient $G$ un groupe alg\'ebrique connexe et $X$ un $G$-torseur. 
Alors $$X(\RA_k)^{\Br_{2/3}(X)} =X(\RA_k)^{\et, \Br} .$$
\end{cor}

Si $G$ est une vari\'et\'e ab\'elienne, on a $\Br_{2/3}(X)=\Br_{1/2}(X)$ (Corollaire \ref{remdefBr2/3}), o\`u $\Br_{1/2}(X)$ est d\'efini dans \cite[pp. 378]{St}.
Donc $X(\RA_k)^{\Br_{1/2}(X)} =X(\RA_k)^{\Br(X)}=X(\RA_k)^{\et, \Br}$, ce qui g\'en\'eralise \cite[Thm. 2]{Cr}, o\`u Creutz montre que $X(\RA_k)^{\Br_{1/2}(X)}=X(\RA_k)^{\Br(X)}$.

\bigskip

Donnons maintenant la structure de l'article.
Au \S 2, on introduit la notion $\Br_{2/3}$ tel que $\Br_{1/2}\sbt \Br_{2/3}\sbt \Br_1$ et on \'etablit le lien de $\Br_{2/3}(X)$ avec l'existence des torseurs sous un $k$-groupe fini commutatif.
Au \S 3, on \'etablit des propri\'et\'es fondamentales du sous-groupe de Brauer $G$-invariant et son lien avec $\Br_{2/3}$.
Au \S 4, inspir\'ee par la suite exacte de Sansuc, on \'etablit la suite exacte correspondante pour un torseur sous un $k$-groupe non lin\'eaire. 
Comme cons\'equence, on montre des suites exactes techniques dans \S 5.
Au \S 6, on \'etablit la m\'ethode de descente des points ad\'eliques orthogonaux aux sous-groupes de Brauer invariants pour un torseur sous un $k$-groupe de type multiplicatif.
Ensuite, on montre le th\'eor\`eme \ref{Thm1} dans \S 7.

\bigskip

\textbf{Conventions et notations}. 

Soit $k$ un corps quelconque de caract\'eristique $0$. On note $\overline{k}$ une cl\^oture alg\'ebrique et $\Gamma_k:=\Gal (\bk/k)$.

Tous les groupes de cohomologie sont des groupes de cohomologie \'etale.
D\'efinissons le faisceau \'etale $\mu_{\infty}:=\mathrm{colim}_n \mu_n$.

Une $k$-vari\'et\'e $X$ est un $k$-sch\'ema s\'epar\'e de type fini. 
 Pour $X$ une telle vari\'et\'e, on note
  $k[X]$ son anneau des fonctions globales,
$k[X]^{\times}$ son groupe des fonctions inversibles,
 $\Pic(X):=H^1_{\text{\'et}}(X,\BG_m)$ son groupe de Picard
 et
$\Br(X):=H_{\text {\'et}}^2 (X, \BG_m)$ son groupe de Brauer. Notons
$$\Br_1 (X) := \Ker [\Br (X) \to\Br (X_ {\bk})]\ \ \text{ et}\ \ \Br_a(X):=\Br_1(X)/\Im\Br(k).$$
Le groupe $\Br_1 (X)$ est le sous-groupe ``alg\'ebrique'' du groupe de Brauer de $X$.
Si $X$ est int\`egre, on note $k(X)$ son  corps des fonctions rationnelles 
et $\pi_1(X,\bar{x})$ (ou $\pi_1(X)$) son groupe fondamental \'etale, o\`u $\bar{x}$ est un point g\'eom\'etrique de $X$.
Soit $\pi_1(X_{\bk})^{ab}$ le quotient maximal ab\'elien de $\pi_1(X_{\bk})$.
Alors $\pi_1(X_{\bk})^{ab} $ est un $\Gamma_k$-module.

Un $k$-groupe alg\'ebrique $G$ est une $k$-vari\'et\'e qui est un $k$-sch\'ema en groupes. 
On note $e_G$ l'unit\'e de $G$ et $G^*$ le groupe des caract\`eres de $G_{\bk}$.
C'est un module galoisien de type fini.

Un $k$-groupe fini $F$ est un $k$-groupe alg\'ebrique qui est fini sur $k$. 
Dans ce cas, $F$ est d\'etermin\'e par le $\Gamma_k$-groupe $F(\bk)$.
Pour toute $k$-vari\'et\'e lisse $X$, on a un isomorphisme canonique (\cite[\S XI.5]{SGA1}):
\begin{equation}\label{BiDprop2.2e}
H^1(\pi_1(X),F(\bk))\iso H^1(X,F)\ \ \ \text{et donc}\ \ \ H^1(X_{\bk},F)\cong \Hom_{cont}(\pi_1(X_{\bk}),F(\bk))/\sim 
\end{equation}
o\`u l'action de $\pi_1(X)$ sur $F(\bk)$ est induite par celle de $\Gamma_k$ et $\sim $ est induite par la conjugaison.

 Soit $G$ un $k$-groupe alg\'ebrique. Une \emph{$G$-vari\'et\'e} $(X,\rho )$ (ou $X$) est une $k$-vari\'et\'e $X$ munie d'une action \`a gauche $G\times_k X\xrightarrow{\rho}X$. 
Un $k$-morphisme de $G$-vari\'et\'es est appel\'e \emph{$G$-morphisme} s'il est compatible avec l'action de $G$.

\section{Pr\'eliminaires sur $\Br_{2/3}$}

Dans toute cette section,  $k$ est un corps quelconque de caract\'eristique $0$. 
Sauf  mention explicite du contraire,  une vari\'et\'e est une $k$-vari\'et\'e.

Soient $X$ une vari\'et\'e lisse et $\pi: X\to \Spec\ k$.
Dans cette section, on suit l'id\'ee de Stoll (la notion de $\Br_{1/2}(X)$ pour $X$ projective dans \cite[pp. 378]{St}) et 
on d\'efinit $\Br_{2/3}(X)$ (d\'efinition \ref{DefiBr2/3}) dans le cas plus g\'en\'eral (on a $\Br_{1/2}\sbt \Br_{2/3}\sbt \Br_1$). 
Ensuite, on \'etablit le lien de $\Br_{2/3}(X)$ avec l'existence des torseurs sous un $k$-groupe fini commutatif (proposition \ref{sec2prop1}) et avec $H^2(X,\mu_{\infty})$ (proposition \ref{lem21.2}).

On d\'efinit:
\begin{equation}
H^2_1(X,\mu_{\infty}):=\Ker(H^2(X,\mu_{\infty})\to H^2(X_{\bk},\mu_{\infty}) ).
\end{equation}
Par d\'efinition, $H^2_1(X,\mu_{\infty})$ est fonctoriel en $X$.

En fait, on peut d\'efinir 
\begin{equation}\label{DefiBr2/3e1}
\Br_{2/3}(X):=\Im(H^2_1(X,\mu_{\infty})\sbt H^2(X,\mu_{\infty})\to \Br(X)) .
\end{equation}
Mais, pour faire lien avec le type d'un torseur, on donne la d\'efinition \ref{DefiBr2/3} ci-dessous et montre que les deux d\'efinitions sont le m\^emes dans la proposition \ref{lem21.2}.

\medskip

Dans \cite{HS13}, Harari et Skorobogatov d\'efinissent 
$$KD(X):=(\tau_{\leq 1}R\pi_*\BG_m)[1]\in D^{[-1,0]}(k),\ \ \   KD'(X):=\mathrm{Cone}(\BG_{m,k}[1]\to KD(X))$$
 et  \'etablissent la suite exacte:
\begin{equation}\label{sec2e1}
H^1(k,S)\to H^1(X,S)\xrightarrow{\chi} \Hom_{D^+(k)}(S^*,KD'(X))\xrightarrow{\partial} H^2(k,S)
\end{equation} 
pour tout $k$-sch\'ema en groupe de type multiplicatif $S$.
L'homomorphisme $\chi$ est appel\'e \emph{le type}.

D'apr\`es \cite[P. 5 Remark]{HS13}, on a 
$$\Br_1(X)\cong H^1(k,KD(X))\ \ \  \text{et} \ \ \ \CH^0(KD'(X))\cong \CH^0(KD(X))\cong \Pic(X_{\bk}).$$
Soient $\Pic(X_{\bk})_{tor}$ (resp. $\Pic(X_{\bk})_{free}$) le sous-groupe torsion maximal (resp. le quotient libre maximal) de $\Pic(X_{\bk})$.
 Ainsi on a une suite exacte de modules galoisiens:
$$0\to \Pic(X_{\bk})_{tor}\to \Pic(X_{\bk})\to \Pic(X_{\bk})_{free}\to 0.$$
Ceci induit
\begin{equation}\label{sec2e2}
\phi: KD(X)\to \Pic(X_{\bk})_{free}, \ \ \ \phi': KD'(X)\to \Pic(X_{\bk})_{free}
\end{equation}
et $\phi_*: H^1(k,KD(X))\to H^1(k,\Pic(X_{\bk})_{free}).$

\begin{defi}\label{DefiBr2/3}
\emph{Le sous-groupe $\Br_{2/3}(X)\sbt \Br(X)$}  d'une vari\'et\'e lisse $X$ est:

(i) si $X$ est connexe, le groupe $\Br_{2/3}(X):=\Ker(\phi_*)\sbt  \Br_1(X);$

(ii) en g\'en\'eral, le groupe $\Br_{2/3}(X)\sbt \Br(X)$ des \'el\'ements $\alpha$
 v\'erifiant $\alpha|_{X'}\in \Br_{2/3}(X')$ pour toute composante connexe $X'$ de $X$.
\end{defi}

Par d\'efinition, $\Br_{2/3}(X)$ est fonctoriel en $X$.
En fait, il contr\^ole l'existence des torseurs sous un $k$-groupe fini commutatif (la proposition \ref{sec2prop1} suivante).

\begin{prop}\label{sec2prop1}
Soient $k$ un corps de nombres, $S$ un $k$-groupe fini commutatif et $X$ une $k$-vari\'et\'e lisse g\'eom\'e\-triquement int\`egre.
Si $X(\RA_k)^{\Br_{2/3}(X)}\neq \emptyset$, alors $\chi$ dans (\ref{sec2e1}) est surjectif. 
\end{prop}

\begin{proof}
Pour tout $\lambda\in \Hom_{D^+(k)}(S^*,KD'(X)),$ on a $\phi'\circ \lambda\in \Hom_k(S^*,\Pic(X_{\bk})_{free})=0$, o\`u $\phi'$ est d\'efini dans (\ref{sec2e2}).
Donc, pour tout $a\in H^1(k,S^*)$, on a $\lambda_*(a)\in \Br_{2/3}(X)/\Im\Br(k)$.
D'apr\`es \cite[Thm. 3.5]{HS13}, on a $\partial(\lambda)\in \Sha^2(k,S)$ et  $<\partial(\lambda),a>_{PT}=0$ pour tout $a\in \Sha^1(k,S^*)$, 
o\`u $<,>_{PT}: \Sha^2(k,S)\times \Sha^1(k,S^*)\to \BQ/\BZ$ est l'accouplement de Poitou-Tate, qui est non d\'eg\'en\'er\'e. 
Donc $\partial(\lambda)=0$ et $\lambda\in \Im(\chi)$.
\end{proof}

\begin{cor}\label{remdefBr2/3}
Soit $X$ une vari\'et\'e projective lisse g\'eom\'e\-triquement int\`egre. 
Si le groupe de N\'eron-Severi $\mathrm{NS}(X_{\bk})$ est sans torsion (par exemple, si $X$ est un torseur sous une vari\'et\'e ab\'elienne), 
alors $\Br_{2/3}(X)=\Br_{1/2}(X)$, o\`u $\Br_{1/2}(X)$ est d\'efini dans \cite[pp. 378]{St}).
\end{cor}

\begin{proof}
Dans ce cas, $\Pic^0(X_{\bk})_{tor}=\Pic(X_{\bk})_{tor}$ et $\Pic^0(X_{\bk})_{free}$ est uniquement divisible. 
Donc $H^1(k,\Pic^0(X_{\bk})_{free})=0$ et 
$$\Im(H^1( k,\Pic^0(X_{\bk}))\to H^1(k,\Pic(X_{\bk})))=\Ker(H^1(k,\Pic(X_{\bk}))\to H^1(k,\Pic(X_{\bk})_{free})),$$ 
d'o\`u le r\'esultat.
\end{proof}

L'inclusion canonique $\mu_{\infty}\sbt \BG_m$ induit 
 $$\psi: \tau_{\leq 1}R\pi_*\mu_{\infty}\to \tau_{\leq 1}R\pi_*\BG_m \in D^+(k)$$
 et donc $\psi_{*,2}:\  H^2_1(X,\mu_{\infty})\cong H^2(k,\tau_{\leq 1}R\pi_*\mu_{\infty})\to H^2(k, \tau_{\leq 1}R\pi_*\BG_m )\cong \Br_1(X)\sbt \Br(X)$.

\begin{prop}\label{lem21.2}
Soit $X$ une vari\'et\'e lisse g\'eom\'e\-triquement int\`egre.
Alors $\Br_{2/3}(X)=\Im(\psi_{*,2} )$ et on a une suite exacte:
$$\Pic(X)\to (\Pic(X_{\bk})_{free,ndiv})^{\Gamma_k} \to H^2_1(X,\mu_{\infty})\xrightarrow{\psi_{*,2}} \Br_{2/3}(X)\to 0,$$
o\`u $\Pic(X_{\bk})_{free,div}\sbt \Pic(X_{\bk})_{free}$ est le sous-groupe divisible maximal et 
$$\Pic(X_{\bk})_{free,ndiv}:=\Pic(X_{\bk})_{free}/\Pic(X_{\bk})_{free,div}.$$
\end{prop}

\begin{proof}
 Soit $D:=\mathrm{Cone} (\psi)$. Par d\'efinition, $\CH^i(D)=0$ pour $i\neq -1,0,1$.
 On a un diagramme commutatif de suites exactes (compatible pour tout $n$):
 $$\xymatrix{\mu_n\ar[r]\ar[d]&\bk[X]^{\times}\ar[r]^{n\cdot}\ar[d]^=&\bk[X]^{\times}\ar[r]\ar[d]^{m\cdot }&H^1(X_{\bk},\mu_n)\ar[r]\ar[d]&\Pic(X_{\bk})\ar[d]^=&&\\
 \mu_{mn}\ar[r]\ar[d]&\bk[X]^{\times}\ar[r]^{mn\cdot}\ar[d]^=&\bk[X]^{\times}\ar[r]\ar[d]&H^1(X_{\bk},\mu_{mn})\ar[r]\ar[d]&\Pic(X_{\bk})\ar[d]^=&&\\
 \mu_{\infty}\ar@{^{(}->}[r]&\bk[X]^{\times}\ar[r]&\CH^0(D)\ar[r]&H^1(X_{\bk},\mu_{\infty})\ar[r]&\Pic(X_{\bk})\ar[r]&\CH^1(D)\ar[r]&0
 }$$
 Alors $\CH^{-1}(D)=0$, $\CH^1(D)\cong \Pic(X_{\bk})_{free}$.
 Puisque le passage \`a la limite directe est exact, on a $\CH^0(D)\cong \underrightarrow{lim}_{n\in \BZ_{>0}}I_n$ avec $I_n=\bk[X]^{\times}$ pour tout $n$ et 
 $I_n\to I_{mn}: \bk[X]^{\times}\to \bk[X]^{\times}: x\mapsto x^m.$
 Donc $\CH^0(D)\cong \bk[X]^{\times}\otimes \BQ$ et $H^i(k,\CH^0(D))=0$ pour $i\geq 1$.
 Donc 
 $$H^1(k,D)\cong (\Pic(X_{\bk})_{free})^{\Gamma_k}\ \ \ \text{et}\ \ \ H^2(k,D)\cong H^1(k,\Pic(X_{\bk})_{free}).$$
 
 En appliquant $H^i(k,-)$ \`a $$\tau_{\leq 1}R\pi_*\mu_{\infty}\to \tau_{\leq 1}R\pi_*\BG_m\to D\xrightarrow{+1},$$ nous obtenons $\Im(\psi_{*,2})=\Ker(\phi_*)$ et la suite exacte:
 $$\Pic(X)\xrightarrow{\phi_1} (\Pic(X_{\bk})_{free})^{\Gamma_k} \to H^2_1(X,\mu_{\infty})\xrightarrow{\psi_{*,2}} \Br_{2/3}(X)\to 0.$$
Puisque $\Pic(X_{\bk})_{free,div}$ est uniquement divisible, le groupe $(\Pic(X_{\bk})_{free,div})^{\Gamma_k}$ est aussi uniquement divisible et on a une suite exacte:
$$0\to (\Pic(X_{\bk})_{free,div})^{\Gamma_k}\to (\Pic(X_{\bk})_{free})^{\Gamma_k}\to (\Pic(X_{\bk})_{free,ndiv})^{\Gamma_k}\to 0.$$
Puisque $H^2_1(X,\mu_{\infty})$ est un groupe de torsion, on a $(\Pic(X_{\bk})_{free,div})^{\Gamma_k}\sbt \Im(\phi_1)$, d'o\`u le r\'esultat.
\end{proof}

\begin{cor}\label{lem21.2cor1}
Soit $X$ une vari\'et\'e lisse g\'eom\'e\-triquement int\`egre.
Soit $K/k$ une extension de corps telle que $k$ soit alg\'ebriquement clos dans $K$.
Alors $$\Br_{2/3}(X)/\Im \Br(k)\to \Br_{2/3}(X_K)/\Im \Br(K)$$ est injectif.
\end{cor}

\begin{proof}
Puisque $H^2(k,\mu_{\infty})\cong \Br(k)$, la proposition \ref{lem21.2} donne un diagramme de suites exactes:
$$\xymatrix{(\Pic(X_{\bk})_{free,ndiv})^{\Gamma_k}\ar[r]\ar[d]^{\phi_1} & H^2_1(X,\mu_{\infty})/\Im H^2(k,\mu_{\infty})\ar[r]\ar[d]^{\phi_2}&
\Br_{2/3}(X)/\Im \Br(k)\ar[r]\ar[d]^{\phi_3}&0\\
(\Pic(X_{\bar{K}})_{free,ndiv})^{\Gamma_K}\ar[r]&H^2_1(X_K,\mu_{\infty})/\Im H^2(K,\mu_{\infty})\ar[r] &\Br_{2/3}(X_K)/\Im \Br(K)\ar[r] &0.
}$$
La suite spectrale $E_2^{i,j}:=H^i(k,H^j(X_{\bk},\mu_{\infty}))\Rightarrow H^{i+j}(X,\mu_{\infty})$ induit une suite exacte
$$H^2(k,\mu_{\infty})\to H^2_1(X,\mu_{\infty})\to H^1(k,H^1(X_{\bk},\mu_{\infty})).$$
Puisque $H^1(X_{\bk},\mu_{\infty})\cong H^1(X_{\bar{K}},\mu_{\infty})$ et l'homomorphisme $\Gal(\bar{K}/K)\to \Gal(\bar{k}/k)$ est surjectif,
l'homomorphisme de inflation $H^1(k,H^1(X_{\bk},\mu_{\infty}))\to H^1(K,H^1(X_{\bar{K}},\mu_{\infty}))$ est injectif 
et donc $\phi_2$ est injectif.
Puisque $\Coker(\Pic(X_{\bk})\to \Pic(X_{\bar{K}}))$ est uniquement divisible, on a 
$ \Pic(X_{\bk})_{free,ndiv}\cong \Pic(X_{\bar{K}})_{free,ndiv}$ et donc $\phi_1$ est un isomorphisme.
Ainsi $\phi_3$ est injectif.
\end{proof}

\bigskip

Soit $G$ un groupe alg\'ebrique connexe. Notons $\pi: G\to \Spec\ k$.
On d\'efinit: $$H^1_e(G,\mu_{\infty}):=\Ker(H^1(G,\mu_{\infty})\xrightarrow{e^*} H^1(k,\mu_{\infty})) ,$$
$$ H^2_{1,e}(G,\mu_{\infty}):=\Ker(H^2_1(G,\mu_{\infty})\xrightarrow{e^*} H^2(k,\mu_{\infty}))$$
et $$\Br_{2/3,e}(G):=\Ker(\Br_{2/3}(G)\xrightarrow{e^*} \Br(k)).$$

Le point $e_G\in G(k)$ induit des isomorphismes: $\Br_{2/3}(G)\cong \Br(k)\oplus \Br_{2/3,e}(G)$,
$$ H^1(G,\mu_{\infty})\cong H^1(k,\mu_{\infty})\oplus H^1_e(G,\mu_{\infty})\ \ \ \text{et}\ \ \   H^2_1(G,\mu_{\infty})\cong H^2(k,\mu_{\infty})\oplus H^2_{1,e}(G,\mu_{\infty}).$$

\begin{cor}\label{prop21.2lem}
On a  $H^1_e(G,\mu_{\infty})\cong H^1(G_{\bk},\mu_{\infty})^{\Gamma_k}$,  $ H^2_{1,e}(G,\mu_{\infty})\cong H^1(k,H^1(G_{\bk},\mu_{\infty}))$
 et  une suite exacte: 
$$\Pic(G)\to  (\Pic(G_{\bk})_{free,ndiv})^{\Gamma_k} \to H^2_{1,e}(G,\mu_{\infty}) \to \Br_{2/3,e}(G)\to 0.$$
\end{cor}

\begin{proof}
Puisque $R^0\pi_*\mu_{\infty}\cong \mu_{\infty}$ et $R^1\pi_*\mu_{\infty}\cong H^1(G_{\bk},\mu_{\infty})$,
le point $e_G\in G(k)$ induit 
$$\tau_{\leq 1}R\pi_*\mu_{\infty}\cong \mu_{\infty}\oplus \left(H^1(G_{\bk},\mu_{\infty})[-1]\right).$$
La suite spectrale $E_2^{p,q}:=H^p(k,R^q\pi_*\mu_{\infty})\Rightarrow H^{p+q}(G,\mu_{\infty})$ donne les isomorphismes dans l'\'enonc\'e.
Puisque $H^2(k,\mu_{\infty})\cong \Br(k)$, la suite exacte d\'ecoule de la proposition \ref{lem21.2}.
\end{proof}

D'apr\`es la formule de K\"unneth (cf. \cite[Prop. 2.6]{C5}), pour toute vari\'et\'e lisse g\'eom\'e\-triquement int\`egre $X$, l'inclusion $X\to G\times X: x\mapsto (e_G,x)$ induit
$$ \tau_{\leq 1}R\pi_{G\times X,*}\mu_{\infty}\cong   (\tau_{\leq 1}R\pi_{X,*}\mu_{\infty})\oplus  \left(H^1(G_{\bk},\mu_{\infty})[-1]\right),$$
o\`u $\pi_{G\times X}: G\times X\to \Spec\ k$ et $\pi_X: X\to \Spec\ k$.
 La suite spectrale $$E_2^{p,q}:=H^p(k,R^q\pi_{G\times X,*}\mu_{\infty})\Rightarrow H^{p+q}(G\times X,\mu_{\infty})$$ induit deux isomorphismes:
\begin{equation}\label{prop21.1e2}
(p_1^*,p_2^*):\ H^1_e(G,\mu_{\infty})\oplus H^1(X,\mu_{\infty}) \cong H^1(G\times X, \mu_{\infty}) 
\end{equation}
et
\begin{equation}\label{prop21.1e1}
(p_1^*,p_2^*):\ H^2_{1,e}(G,\mu_{\infty})\oplus H^2_1(X,\mu_{\infty}) \cong H^2_1(G\times X, \mu_{\infty}) .
\end{equation}
o\`u $p_1: G\times X\to G$, $p_2: G\times X\to X$ sont deux projections.

S'il existe une action $\rho: G\times X\to X$, alors (\ref{prop21.1e2}) induit un homomorphisme
\begin{equation}\label{prop21.1e3}
\lambda_{H^1}: H^1(X,\mu_{\infty}) \to H^1_e(G,\mu_{\infty})
\end{equation}
tel que $p_1^*\circ \lambda_{H^1}=\rho^*-p_2^* : H^1(X,\mu_{\infty})\to H^1(G\times X,\mu_{\infty})$.

\section{Propri\'et\'es du sous-groupe invariant}

Dans toute cette section,  $k$ est un corps quelconque de caract\'eristique $0$. 
Sauf  mention explicite du contraire,  une vari\'et\'e est une $k$-vari\'et\'e.

Soient $G$ un groupe alg\'ebrique connexe et $(X,\rho)$ une $G$-vari\'et\'e lisse g\'eom\'etriquement connexe.
 Soient $p_1: G\times X\to G$, $p_2: G\times X\to X$ deux projections.
 Dans cette section, on \'etablit les propri\'et\'es fondamentales de $H^2_G(X,\mu_{\infty})$ (proposition \ref{prop21.1}) et de $\Br_G'(X)$ (proposition \ref{prop21.2}).

\subsection{Groupe alg\'ebrique sur un corps alg\'ebriquement clos} 
Dans cette partie, pour tout groupe alg\'ebrique connexe $G$ sur un corps alg\'ebriquement clos de caract\'eristique $0$, en utilisant la cohomologie singuli\`ere de certain groupes de Lie (le th\'eor\`eme \ref{thmCartan1} ci-dessous), on montre que $H^2_G(G,\mu_{\infty})=0$ (le corollaire \ref{corCartan1} (i)).

Rappelons le th\'eor\`eme classique de \'Elie Cartan.

\begin{thm}(\'Elie Cartan)\label{thmCartan1}
Soit $H$ un groupe de Lie connexe. Alors $\pi_2^{top}(H)=0$ et les cohomologies singuli\`eres $H^i_{sing}(H,\BQ/\BZ)\cong H^i(\pi_1^{top}(H),\BQ/\BZ)$ pour $i=1,2$.
\end{thm}

\begin{proof}
Pour $\pi_2^{top}(H)$, d'apr\`es \cite[Thm. 6]{Iw} (cf. \cite[pp. 116, Rem. 2]{Mil}), 
tout groupe de Lie connexe contient un sous-groupe de Lie compact comme une r\'etraction par d\'eformation. 
Donc on peut supposer que $H$ est compact et, dans ce cas, l'\'enonc\'e ci-dessus est exactement le \cite[Prop. V.7.5]{BrD}.

Pour $H^i_{sing}(H,\BQ/\BZ)$, soit $U\to H$ un rev\^etement universel. Alors $\pi_1^{top}(U)=\pi_2^{top}(U)=0$. L'application de Hurewicz (cf. \cite[Thm. A.5]{Sr}) et le th\'eor\`eme des coefficients universels (cf. \cite[Thm. A.15]{Sr}) impliquent $H^1_{sing}(U,\BQ/\BZ)=H^2_{sing}(U,\BQ/\BZ)=0$.
Ensuite la suite spectrale de Leray-Serre (cf. \cite[Thm. A.24]{Sr}) implique l'\'enonc\'e.
\end{proof}

Soit $G$ un $k$-groupe alg\'ebrique connexe.
Notons $p_1: G\times X\to G$, $p_2: G\times X\to X$ deux projections et $m: G\times G\to G$ la multiplication.

Dans le cas o\`u $k=\BC$, d\'efinissons:
$$H^2_{sing,G}(G(\BC),\BQ/\BZ):=\{ b\in H^2_{sing}(G(\BC),\BQ/\BZ)\ :\ (m^*(b)-p_2^*(b))\in p_1^*H^2_{sing}(G(\BC),\BQ/\BZ) \}.$$

\begin{prop}\label{propCartan1}
Soit $G$ un $\BC$-groupe alg\'ebrique connexe. 
Alors $H^2_{sing,G}(G(\BC),\BQ/\BZ)=0$. 
\end{prop}

\begin{proof}

Soit $\pi:=\pi_1^{\mathrm{top}}(G)$. Pour tout $\BZ$-module $M$, notons $M^D:=\Hom(M,\BQ/\BZ)$.
 D'apr\`es le th\'eor\`eme \ref{thmCartan1} et le lemme \ref{prop21.1-lem} ci-dessous, on a l'inclusion naturelle 
 $$H^2_{sing}(G(\BC),\BQ/\BZ)\cong H^2(\pi,\BQ/\BZ) \sbt (\pi\otimes_{\BZ}\pi)^D,$$
 et donc l'inclusion natuelle $H^2_{sing}((G\times G)(\BC),\BQ/\BZ)\sbt ((\pi\oplus \pi)\otimes_{\BZ}(\pi\oplus\pi))^D$.

Notons $m_*,p_{1,*},p_{2,*}: \pi^{top}_1(G\times G)\to \pi^{top}_1(G)$ les homomorphismes induits par $m,p_1,p_2$.
Puisque $m|_{e_G\times G}=id_G$ et $m|_{G\times e_G}=id_G$, on a $m_*: \pi\oplus \pi\to \pi: (a,b)\to a+b$ et 
$$H^2_{sing,G}(G(\BC),\BQ/\BZ)=\Ker(m^*-p_1^*-p_2^*:\ H^2_{sing}(G(\BC),\BQ/\BZ)\to H^2_{sing}((G\times G)(\BC),\BQ/\BZ)). $$
L'homomorphisme $m_*-p_{1,*}-p_{2,*}: (\pi\oplus \pi)\otimes_{\BZ}(\pi\oplus\pi)\to \pi\otimes \pi$ envoie $(a,b)\otimes (c,d)$ sur
$$(a+b)\otimes (c+d)-a\otimes c-b\otimes d=a\otimes d+b\otimes c.  $$
Donc l'application $m_*-p_{1,*}-p_{2,*}$ est surjective et 
$$(m_*-p_{1,*}-p_{2,*})^D: (\pi\otimes_{\BZ} \pi)^D  \to ((\pi\oplus \pi)\otimes_{\BZ}(\pi\oplus\pi))^D$$
 est injective.
  Ainsi $m^*-p_1^*-p_2^*$ est injective et $H^2_{sing,G}(G(\BC),\BQ/\BZ)=0$.
\end{proof}

\begin{lem}\label{prop21.1-lem}
(1) Soient $\pi^1$, $\pi^2$ deux groupes.
Alors $$H^2(\pi^1,\BQ/\BZ)\oplus H^2(\pi^2,\BQ/\BZ) \oplus [((\pi^1)^{ab}\otimes (\pi^2)^{ab})^D] \cong H^2(\pi^1\times \pi^2,\BQ/\BZ),$$
o\`u $(-)^{ab}$ est le quotient ab\'elien maximal.

(2) Soit $\pi$ un groupe ab\'elien de type fini. Alors $H^1(\pi,\BQ/\BZ)\cong \pi^D$ et on a une inclusion naturelle 
$H^2(\pi,\BQ/\BZ)\sbt (\pi\otimes_{\BZ}\pi)^D.$
\end{lem}

\begin{proof}
Pour tout $\BZ$-module $M$, notons $M^D:=\Hom(M,\BQ/\BZ)$.

Soit $\pi$ un groupe.
D'apr\`es \cite[Thm. 6.1.11]{We}, on a $H_1(\pi,\BZ)\cong \pi/[\pi,\pi]\cong (\pi)^{ab}$. 
D'apr\`es \cite[Exer. 6.1.5]{We}, pour tout $i\geq 0$, on a $H^i(\pi,\BQ/\BZ)\cong H_i(\pi,\BZ)^D$.
Alors $H^1(\pi,\BQ/\BZ)\cong (\pi^{ab})^D$.

D'apr\`es \cite[Prop. 6.1.13]{We}, on a: $H_1(\pi^1,\BZ)\oplus H_1(\pi^2,\BZ)\cong H_1(\pi^1\times \pi^2,\BZ)$ et
$$H_2(\pi^1,\BZ)\oplus H_2(\pi^2,\BZ)\oplus [H_1(\pi^1,\BZ)\otimes_{\BZ}H_1(\pi^2,\BZ)] \cong H_2(\pi^1\times \pi^2,\BZ).$$
Donc $H^1(\pi^1,\BQ/\BZ)\oplus H^1(\pi^2,\BQ/\BZ)\cong H^1(\pi^1\times \pi^2,\BQ/\BZ)$ et
$$H^2(\pi^1,\BQ/\BZ)\oplus H^2(\pi^2,\BQ/\BZ) \oplus [H_1(\pi^1,\BZ)\otimes_{\BZ}H_1(\pi^2,\BZ)]^D \cong H^2(\pi^1\times \pi^2,\BQ/\BZ), $$
 d'o\`u l'on d\'eduit (1).
 
 Pour (2), on a $\pi\cong \oplus_{i=1}^rM_i$ avec $M_i\cong \BZ/n_i$ pour certain $n_i\in \BZ_{\geq 0}$ (on a $M_i\cong \BZ$ si $n_i=0$).
Ainsi $(\pi\otimes \pi)^D\cong \oplus_{1\leq i\leq r,1\leq j\leq r} [(M_i\otimes M_j)^D]$.
 Puisque $H^2(\BZ/n,\BQ/\BZ)=0$ pour tout $n\in \BZ_{\geq 0}$ (\cite[Thm. 6.2.2 et Cor. 6.2.7]{We}), d'apr\`es (1), on a
 $$H^2(\pi,\BQ/\BZ)\cong \oplus_{1\leq i<j\leq r}[(M_i\otimes M_j)^D].$$
 On a l'inclusion canonique $\oplus_{ i<j}[(M_i\otimes M_j)^D]\sbt \oplus_{i,j} [(M_i\otimes M_j)^D]: (m_{i,j})_{i<j}\mapsto (n_{i,j})_{i,j}$ o\`u $n_{i,j}=n_{j,i}=m_{i,j}$ pour $i<j$ et $n_{i,i}=0$. Ceci donne (2).
\end{proof}

\begin{cor}\label{corCartan1}
Soit $k$ un corps alg\'ebriquement clos de caract\'eristique $0$. Alors:

(i) pour tout $k$-groupe alg\'ebrique connexe $G$, on a $H^2_G(G,\mu_{\infty})=0$;

(ii) pour toute suite exacte $1\to G_1\to G_2\to T\to 1$ de $k$-groupes alg\'ebriques connexes avec $T$ un tore, on a une suite exacte canonique:
$$0\to H^1(T,\BQ/\BZ)\to H^1(G_2,\BQ/\BZ)\to H^1(G_1,\BQ/\BZ)\to 0 .$$
\end{cor}

\begin{proof}
Pour (i), par le passage \`a la limite (cf. \cite[\S 1.10]{Fu}), il existe un sous-corps $k_0\subset k$ et un $k_0$-groupe alg\'ebrique connexe $G_0$ tels que $k_0$ soit une extension de type fini sur $\BQ$ et $G_{0,k}\cong G$. 
Choisissons une immersion $\bar{k}_0\subset \BC$.
 D'apr\`es \cite[Thm. III.3.12]{Mi80} et \cite[Cor. VI.4.3]{Mi80}, on a des isomorphismes naturels: 
 $$H^2(G,\BQ/\BZ)\cong H^2(G_{0,\bar{k}_0},\BQ/\BZ) \cong H^2(G_{0,\BC},\BQ/\BZ)\cong  H^2_{sing}(G_{0,\BC}(\BC),\BQ/\BZ).$$
 Puisque $\mu_{\infty}=\BQ/\BZ$, on a $H^2_G(G,\mu_{\infty})=H^2_{sing,G_{0,\BC}}(G_{0,\BC}(\BC),\BQ/\BZ).$
Donc on peut supposer que $k=\BC$ et il suffit de montrer que $H^2_{sing,G}(G(\BC),\BQ/\BZ)=0$. Ainsi la proposition  \ref{propCartan1} implique (i).

 Pour (ii), par le m\^eme argument comme ci-dessus, on peut supposer que $k=\BC$ et il suffit de montrer que 
 \begin{equation}\label{eqcorCartan1-e}
  0\to H^1_{sing}(T(\BC),\BQ/\BZ)\to H^1_{sing}(G_2(\BC),\BQ/\BZ)\to H^1_{sing}(G_1(\BC),\BQ/\BZ)\to 0  
  \end{equation}
 est exacte.
 On a une suite exacte $0\to \pi_1^{top}(G_1)\to \pi_1^{top}(G_2)\to \pi_1^{top}(T)\to 0$ de groupes ab\'eliens, qui admet une section (non canonique), parce que $\pi_1^{top}(T)\cong \BZ^n$. 
 D'apr\`es le th\'eor\`eme \ref{thmCartan1}, on a 
 $$H^1_{sing}(G(\BC),\BQ/\BZ)=H^1(\pi_1^{top}(G),\BQ/\BZ)=\Hom(\pi_1^{top}(G),\BQ/\BZ)$$
  pour $G=G_1,G_2$ ou $T$, d'o\`u on obtient (\ref{eqcorCartan1-e}) et (ii).
\end{proof}

\subsection{Le cas g\'en\'eral}

\begin{prop}\label{prop21.1}
Soient $G$ un groupe alg\'ebrique connexe et $(X,\rho)$ une $G$-vari\'et\'e lisse g\'eom\'etriquement connexe.
Alors  

(1)  on a: $H^2_1(X,\mu_{\infty})\sbt H^2_G(X,\mu_{\infty})$;

(2) si $X$ est un $G$-torseur, on a  $H^2_1(X,\mu_{\infty})=H^2_G(X,\mu_{\infty})$;

(3) il existe un homomorphisme unique (\emph{l'homomorphisme de Sansuc}) $$\lambda: H^2_G(X,\mu_{\infty})\to H^2_{1,e}(G,\mu_{\infty})$$
tel que $\rho^*-p_2^*=p_1^*\circ \lambda$;

(4) pour tout $x\in X(k)$, on a $\lambda =\rho_x^*-i_x^*: H^2_G(X,\mu_{\infty})\to H^2_1(G,\mu_{\infty})$,
o\`u $i_x:G\to X: g\mapsto x$ et $\rho_x: G\to X: g\mapsto g\cdot x$.
\end{prop}

\begin{proof}
D'apr\`es (\ref{prop21.1e1}), en utilisant $\rho |_{e_G\times X}=id_X$, on a (1).

Pour (2), par fonctorialit\'e, il suffit de montrer que $H^2_{G_{\bk}}(X_{\bk},\mu_{\infty})=0$.
On peut supposer que $k=\bk$. 
Dans ce cas, $X\cong G$, $\mu_{\infty}\cong \BQ/\BZ$ et $\rho=m: G\times G\to G$ la multiplication.
D'apr\`es le corollaire \ref{corCartan1} (i), $H^2_G(G,\BQ/\BZ)=0$, d'o\`u l'on d\'eduit (2).

Pour (3), d'apr\`es (\ref{prop21.1e1}), $p_1^*|_{H^2_{1,e}(G,\mu_{\infty})}$ est injective et donc il suffit de montrer que 
$$(\rho^*-p_2^*)(H^2_G(X,\mu_{\infty})) \sbt p_1^*(H^2_{1,e}(G,\mu_{\infty})).$$
Puisque $(\rho^*-p_2^*)|_{e_G\times X}=0$, d'apr\`es (\ref{prop21.1e1}), il suffit de montrer que 
$$(\rho^*-p_2^*)(H^2_G(X,\mu_{\infty})) \sbt H^2_1(G\times X,\mu_{\infty}).$$
On peut supposer que $k=\bk$. 
Un point $x\in X(k)$ induit un morphisme $G\xrightarrow{\rho_x} X:\ g\mapsto g\cdot x$. 
Alors on a un diagramme commutatif:
$$\xymatrix{H^2(X,\mu_{\infty})\ar[r]^-{p_2^*-\rho^*}\ar[d]^{\rho_x^*}&H^2(G\times X,\mu_{\infty})\ar[d]^{(id_G\times \rho_x)^*}&H^2(G,\mu_{\infty})\ar[d]^=\ar@{_(->}[l]_-{p_1^*}\\
H^2(G,\mu_{\infty})\ar[r]_-{p_{2,G}^*-m^*}&H^2(G\times G,\mu_{\infty})&H^2(G,\mu_{\infty})\ar@{_(->}[l]^-{p_{1,G}^*}.
}$$
D'apr\`es (1), $\Im(p_{2,G}^*-m^*)\cap \Im(p_{1,G}^* )=0$ et donc $\Im(p_2^*-\rho^*)\cap \Im(p_1^* )=0$, d'o\`u (3).

Pour (4), pour tout $\alpha\in H^2_G(X,\mu_{\infty})$, on obtient
 $ (\rho^*-p_2^*)(\alpha)|_{G\times x}=(\rho_x^*-i_x^*)(\alpha)$.
\end{proof}

La suite spectrale $E_2^{i,j}:=H^i(k,H^j(X_{\bk},\mu_{\infty}))\Rightarrow H^{i+j}(X,\mu_{\infty})$ induit une suite exacte
\begin{equation}\label{prop21.1e5}
H^2(k,\mu_{\infty})\to H^2_1(X,\mu_{\bk})\xrightarrow{\partial} H^1(k,H^1(X_{\bk},\mu_{\infty})) \to \Ker(H^3(k,\mu_{\infty})\to H^3(X,\mu_{\infty})). 
\end{equation}
Par les d\'efinitions de $\lambda$ ci-dessus et de $\lambda_{H^1}$  dans (\ref{prop21.1e3}), on a un diagramme commutatif:

\begin{equation}\label{prop21.1e4}
\xymatrix{H^2_1(X,\mu_{\infty})\ar@{^{(}->}[r]\ar[d]^{\partial}&H^2_G(X,\mu_{\infty})\ar[r]^-{\lambda}&H^2_{1,e}(G,\mu_{\infty})\ar[d]^=\\
H^1(k,H^1(X_{\bk},\mu_{\infty}))\ar[r]^-{\lambda_{H^1,*}}&H^1(k,H^1(G_{\bk},\mu_{\infty}))&H^2_{1,e}(G,\mu_{\infty})\ar[l]_-{\cong}.
}
\end{equation}

\bigskip

\begin{prop}\label{prop21.2}
Soient $G$ un groupe alg\'ebrique connexe et $(X,\rho)$ une $G$-vari\'et\'e lisse g\'eom\'etriquement connexe.
Alors  

(1)  on a: $\Br_{2/3}(X)\sbt \Br'_G(X)\sbt \Br_G(X)$;

(2) si $X$ est un $G$-torseur, on a  $\Br_{2/3}(X)=\Br'_G(X)$;

(3) il existe un homomorphisme unique (\emph{l'homomorphisme de Sansuc}) $$\lambda_{\Br}: \Br'_G(X)\to \Br_{2/3,e}(G)$$ 
tel que $(\rho^*-p_2^*)=p_1^*\circ \lambda_{\Br}$;

(4) pour tout $x\in X(k)$, on a $\lambda_{\Br} =\rho_x^*-i_x^*: \Br'_G(X)\to \Br_{2/3}(G)$,
o\`u $i_x:G\to X: g\mapsto x$ et $\rho_x: G\to X: g\mapsto g\cdot x$.
\end{prop}

\begin{proof}
D'apr\`es  la proposition \ref{lem21.2}, $\Br_{2/3}(X)=\Im (H^2_1(X,\mu_{\infty})\to \Br(X))$.
Alors la proposition \ref{prop21.1} (1) et (2) impliquent les (1) et (2) ici.
Pour (3), d'apr\`es la proposition \ref{prop21.1} (3), on a $(\rho^*-p_2^*)(\Br'_G(X)) \sbt p_1^*(\Br_{2/3,e}(G))$.
Donc il suffit de montrer que $p_1^*|_{\Br_{2/3,e}(G)}$ est injectif, 
i.e. $p^*_1: \Br_{2/3,e}(G)\to \Br_{2/3,e}(G_K)$ est injectif, o\`u $K:=k(X)$ et $p: G_K\to G$.
Ceci d\'ecoule du corollaire \ref{lem21.2cor1}.
Ensuite, le (3) et la proposition \ref{prop21.1} (4) implique le (4).
\end{proof}

 Pour toute extension de corps $K/k$, et tous $x\in X(K)$, $g\in G(K)$, $\alpha\in \Br'_G(X)$, d'apr\`es la proposition \ref{prop21.2} (3), on a:
\begin{equation}\label{invequa2.3e1}
(g\cdot x)^*(\alpha )= g^*(\lambda_{\Br} (\alpha))+x^*(\alpha)\in \Br(K).
\end{equation}
Alors, dans le cas o\`u $k$ est un corps de nombres, on a:
\begin{equation}\label{invequa2.3}
G(\RA_k)^{\Br_{2/3}(G)}\cdot X(\RA_k)^{\Br'_G(X)}=X(\RA_k)^{\Br'_G(X)}.
\end{equation}

\medskip

Dans le cas o\`u $G$ est lin\'eaire, on a:

\begin{cor}\label{corlinBr1et2/3}
Sous les hypoth\`eses de la proposition \ref{prop21.2}, supposons que $G$ est lin\'eaire.
Alors $\Br_{2/3}(G)= \Br_1(G)\cong H^2_1(G,\mu_{\infty})$ et $\Br_G(X)=\Br'_G(X)$.
\end{cor}

\begin{proof}
Si $G$ est lin\'eaire,  on a $\Pic(G_{\bk})_{free}=0$ et $\Pic(G)$ est fini. Par d\'efinition, on a $\Br_{2/3}(G) = \Br_1(G).$
D'apr\`es le corollaire \ref{prop21.2lem}, on a  $\Br_{2/3}(G)\cong H^2_1(G,\mu_{\infty})$.

D'apr\`es \cite[Lem. 6.6]{S}, $\Pic(G\times X)\cong \Pic(G)\oplus \Pic(X)$ et donc $\Pic(G\times X)\otimes \BQ/\BZ\cong  \Pic(X)\otimes \BQ/\BZ$.
Notons $i_e: X\to G\times X: x\mapsto (e_G,x)$. La suite exacte de Kummer donne un diagramme commutatif de suites exactes
$$ \xymatrix{&&H^2(X,\mu_{\infty})\ar[r]^{l_X}\ar[d]^{\phi_{H^2}}&\Br(X)\ar[d]^{\phi_{\Br}}\ar[r]&0\\
0\ar[r]&\Pic(G\times X)\otimes \BQ/\BZ\ar[r]\ar[d]^{\cong}&H^2(G\times X,\mu_{\infty})\ar[r]^{l_{G\times X}}\ar[d]^{i_{e,H^2}^*}&\Br(G\times X)\ar[r]\ar[d]^{i^*_{e,\Br}}&0\\
0\ar[r]&\Pic(X)\otimes \BQ/\BZ\ar[r]&H^2(X,\mu_{\infty})\ar[r]&\Br(X)\ar[r]&0,
}$$
o\`u $\phi_{H^2}:=(\rho^*-p_2^*)|_{H^2(X,\mu_{\infty})}$ et $\phi_{\Br}:=(\rho^*-p_2^*)|_{\Br(X)}$.
Alors $l_{G\times X}$ induit un isomorphisme $l_0: \Ker(i^*_{e,H^2})\iso \Ker(i^*_{e,\Br})$.
Puisque $\Br_1(G)\cong H^2_1(G,\mu_{\infty})$, on a  
$$p_1^*|_{H^2}(H^2_1(G,\mu_{\infty}))=l_0^{-1} (p_1^*|_{\Br}(\Br_1(G))).$$

 Pour tout $b\in H^2(X,\mu_{\infty})$ avec $a:=l_X(b)\in \Br_G(X)$, 
 on a $\phi_{\Br}(a)\in p_1^*|_{\Br}(\Br_1(G))$ (\cite[Prop. 3.7]{C1}),  $i^*_{e,H^2}(\phi_{H^2}(b))=0$ (puisque $\rho\circ i_e=p_2\circ i_e$),
 et donc $\phi_{H^2}(b)\in p_1^*|_{H^2}(H^2_1(G,\mu_{\infty}))$. 
 Alors $b\in H^2_G(X,\mu_{\infty})$ et $\Br_G(X)\sbt \Br'_G(X)$, d'o\`u le r\'esultat.
\end{proof}

\section{Suites exactes de Sansuc}\label{section4}

Dans toute cette section,  $k$ est un corps quelconque de caract\'eristique $0$. 
Sauf  mention explicite du contraire,  une vari\'et\'e est une $k$-vari\'et\'e.

Le but de cette section est d'\'etablir des suites exactes fondamentales de $H^2_1$, $H^2_G$ et de $\Br'_G$.
Le th\'eor\`eme \ref{sansucthm} est une variante de la suite exacte de Sansuc \cite[Prop. 6.10]{S} en rempla\c{c}ant la cohomologie de $\BG_m$ par la cohomologie de $\mu_{\infty}$, dans le cas o\`u $G$ n'est pas n\'ecessairement lin\'eaire. 
Ensuite, en utilisant ce th\'eor\`eme, on montre que les sous-groupes invariants $H^2_G(X,\mu_{\infty})$ et $\Br'_G(X)$
 sont stables si l'on remplace $X$ par son torseur sous un groupe connexe (proposition \ref{sansucthm-cor1}) 
 ou si l'on remplace $G$ par son extension avec un groupe connexe (proposition \ref{sansucthm-prop}).

La suite exacte longue dans le th\'eor\`eme \ref{sansucthm} suivant est appel\'ee \emph{la suite exacte de Sansuc} de $\mu_{\infty}$.

\begin{thm}\label{sansucthm}
Soient $Z$ une $k$-vari\'et\'e lisse g\'eom\'e\-triquement int\`egre, $G$ un $k$-groupe alg\'ebrique connexe et $f: X\to Z$ un $G$-torseur.
Alors on a une suite exacte, fonctorielle en $(X,Z,f,G)$:
$$0\to H^1(Z,\mu_{\infty})\xrightarrow{f^*} H^1(X,\mu_{\infty})\xrightarrow{\lambda_{H^1}} H^1_e(G,\mu_{\infty})\to H^2(Z,\mu_{\infty})\xrightarrow{f^*} H^2(X,\mu_{\infty})\xrightarrow{\rho^*-p_2^*} H^2(G\times X,\mu_{\infty}),$$
o\`u $\rho: G\times X\to X $ est l'action et $\lambda_{H^1}$ est d\'efini dans (\ref{prop21.1e3}).
\end{thm}

\begin{proof}
On suit l'id\'ee de Sansuc \cite[Prop. 6.10]{S}: on utilise la suite spectrale qui lie la cohomologie \'etale avec la cohomologie de C\v{e}ch (la suite spectrale (\ref{sansucthm-e1}) ci-dessous). 
Alors on fait tout d'abord des calculations sur la cohomologie de C\v{e}ch.

\medskip

{\bf \'Etape 1, pr\'efaisceaux des cohomologies \'etales.} 

Soient 
$\CH^i: U\to H^i(U,\mu_{\infty})$, $\CH^1_0: U\to H^1(\pi_0(U),\mu_{\infty})$ et $\CH^1_a: U\to \Coker(H^1(\pi_0(U),\mu_{\infty})\to H^1(U,\mu_{\infty}))$ des pr\'efaisceaux, 
o\`u $\pi_0(U)$ est le sch\'ema des composantes connexes g\'eom\'e\-triques de $U$. 
Le $\CH^i$ ici est exactement le $\underline{H}^i$ dans \cite[III]{Mi80} (cf. \cite[III, Rem. 1.6 (e)]{Mi80}).
Alors on a une suite exacte de pr\'efaisceaux: 
\begin{equation}\label{sansucthm-e2}
0\to \CH^1_0\to\CH^1\to\CH^1_a\to 0.
\end{equation}
Le rev\^etement fid\`element plat de pr\'esentation finie $f: X\to Z$ induit une suite exacte longue des cohomologies de C\v{e}ch:
\begin{equation}\label{sansucthm-e4}
\cdots \to \check{H}^i(X/Z, \CH^1_0)\to  \check{H}^i(X/Z, \CH^1)\to  \check{H}^i(X/Z, \CH^1_a)\to  \check{H}^{i+1}(X/Z, \CH^1_0)\to \cdots .
\end{equation}

\medskip

{\bf \'Etape 2, calculation de la cohomologie de C\v{e}ch.}

Notons $G^1:=G,$ $ X^1:=X,$ $G^i:=G\times G^{i-1}$ et  $X^i:=X\times_Z X^{i-1}$  pour $i\geq 1$. 
Notons $p^i_j: X^{i+1}\to X^i$ la projection qui enl\`eve la j-i\`eme coordonn\'ee, o\`u $1\leq j\leq i+1$.
Ainsi on a directement $X^i\cong G^{i-1}\times X$. 
Alors la cohomologie de C\v{e}ch d'un pr\'efaisceau $\CF$ est la cohomologie du complexe:
$$ 0\to \CF(X^1)\xrightarrow{p_1^{1,*}-p_2^{1,*}}   \CF(X^2) \xrightarrow{p_1^{2,*}-p_2^{2,*}+p_3^{2,*} }\CF(X^3) \to \cdots .$$
Par exemple, $\check{H}^0(X/Z,\CF)=\ker(p_1^{1,*}-p_2^{1,*})$.

(i) On a $\check{H}^0(X/Z,\CH^0)=\mu_{\infty}(k)$ et $\check{H}^i(X/Z,\CH^0)=0$ pour tout $i\geq 1$. 
Ceci vaut car $\CH^0(X^i)\cong \mu_{\infty}(k)$ pour $i$ et tout $p^{i,*}_j$ est l'identit\'e pour tout $i,j$.

(ii) On a $\check{H}^0(X/Z,\CH^1_0)=H^1(k,\mu_{\infty})$ et $\check{H}^i(X/Z,\CH^1_0)=0$ pour tout $i\geq 1$. Ceci vaut car $\pi_0(X^i)=\Spec\ k$, $\CH^1_0(X^i)\cong H^1(k,\mu_{\infty})$ pour tout $i$ et $p^{i,*}_j$ est l'identit\'e pour tout $i,j$.

(iii) On a $\check{H}^i(X/Z,\CH^1_a)=0$ pour tout $i\geq 2$ et une suite exacte naturelle: 
$$0\to \check{H}^0(X/Z,\CH^1_a)\to \CH^1_a(X)\xrightarrow{\rho^*-p_2^*} \CH^1_a(G)\to \check{H}^1(X/Z,\CH^1_a)\to 0.$$
En fait, d'apr\`es \cite[Lem. 6.12]{S} et le fait $\CH^1_a(k)=0$, il suffit de montrer que 
$$\CH^1_a(X\times G^i)\cong  \CH^1_a(X)\oplus \CH^1_a(G)^{\oplus i}.$$
Ceci est \'etabli dans (\ref{prop21.1e2}).

(iv) On a $\check{H}^i(X/Z,\CH^1)=0$ pour tout $i\geq 2$ et une suite exacte naturelle: 
\begin{equation}\label{sansucthm-e3}
0\to \check{H}^0(X/Z,\CH^1)\to \CH^1(X)\xrightarrow{\rho^*-p_2^*} \CH^1_a(G)\to \check{H}^1(X/Z,\CH^1)\to 0.
\end{equation}
D'apr\`es (ii) et les suites exactes (\ref{sansucthm-e2}) et (\ref{sansucthm-e4}), on a $\check{H}^i(X/Z,\CH^1)\cong \check{H}^i(X/Z,\CH^1_a)$ pour tout $i\geq 1$
 et un diagramme commutatif de suites exactes:
$$\xymatrix{0\ar[r]&\check{H}^0(X/Z,\CH^1_0)\ar[r]\ar[d]^{\cong}&\check{H}^0(X/Z,\CH^1)\ar[r]\ar@{_(->}[d] &\check{H}^0(X/Z,\CH^1_a)\ar[r]\ar@{_(->}[d]^{\phi}&0\\
0\ar[r]&\CH^1_0(X)\ar[r] &\CH^1(X)\ar[r]&\CH^1_a(X)\ar[r]&0.
}$$
D'apr\`es le  lemme du serpent, on a une suite exacte:
$$ 0\to \check{H}^0(X/Z,\CH^1)\to \CH^1(X)\to \coker(\phi)\to 0 .$$
L'\'enonc\'e d\'ecoule de (iii).

\medskip

{\bf \'Etape 3, la suite spectrale.} 

Puisque $H^i(X,\mu_{\infty})\cong H^i_{fppf}(X,\mu_{\infty})$ (\cite[III, Thm. 3.9]{Mi80} pour $\mu_n$) et $f: X\to Z$ est un rev\^etement fid\`element plat de pr\'esentation finie, on a la suite spectrale (\cite[III, Prop. 2.7]{Mi80})
\begin{equation}\label{sansucthm-e1}
E_2^{i,j}=\check{H}^i(X/Z,\CH^j)\Rightarrow H^{i+j}(Z,\mu_{\infty}).
\end{equation}

Dans la suite spectrale (\ref{sansucthm-e1}), d'apr\`es (i) et (iv), $E_2^{i,0}=0$ pour tout $i\geq 1$ et $E^{i,1}_2=0$ pour tout $i\geq 2$.
Alors cette suite spectrale induit un isomorphisme $ \check{H}^0(X/Z,\CH^1)\cong H^1(Z,\mu_{\infty})$ et une suite exacte: 
$$0\to \check{H}^1(X/Z,\CH^1)\to  H^2(Z,\mu_{\infty})\to \check{H}^0(X/Z,\CH^2)\to 0.$$
D'apr\`es (\ref{sansucthm-e3}), on a une suite exacte longue:
$$0\to H^1(Z,\mu_{\infty}) \to \CH^1(X)\xrightarrow{\rho^*-p_2^*} \CH^1_a(G)\to  H^2(Z,\mu_{\infty})\to \check{H}^0(X/Z,\CH^2)\to 0.$$
Puisque $\check{H}^0(X/Z,\CH^2)=\Ker (H^2(X,\mu_{\infty})\xrightarrow{\rho^*-p_2^*} H^2(G\times X,\mu_{\infty}))$,
 il suffit de montrer que l'homomorphisme canonique $l: H^1_e(G,\mu_{\infty})\to \CH^1_a(G)$ est un isomorphisme et que $l\circ \lambda_{H_1}=\rho^*-p_2^*$. 
 Ceci d\'ecoule de la d\'efinition et de (\ref{prop21.1e3}).
\end{proof}

\begin{cor}\label{sansucthm-cor}
Sous les hypoth\`eses du th\'eor\`eme \ref{sansucthm}, on a un diagramme commutatif de suites exactes, fonctoriel en $(X,Z,f,G)$:
$$\xymatrix{ H^1_e(G,\mu_{\infty})\ar[d]^=\ar[r]& H^2(Z,\mu_{\infty})\ar[r]^{f^*}\ar[d]^=& H^2_G(X,\mu_{\infty})\ar[r]^{\lambda}\ar@{^{(}->}[d]& H^2_{1,e}(G,\mu_{\infty})\ar@{^{(}->}[d]^{p_1^*}\\
 H^1_e(G,\mu_{\infty})\ar[r]& H^2(Z,\mu_{\infty})\ar[r]^{f^*} & H^2(X,\mu_{\infty})\ar[r]^-{\rho^*-p_2^*}& H^2(G\times X,\mu_{\infty}).}$$
\end{cor}

\begin{proof}
D'apr\`es (\ref{prop21.1e1}), $p_1^*$ est injective. Le r\'esultat d\'ecoule du th\'eor\`eme \ref{sansucthm} et de la proposition \ref{prop21.1} (3).
\end{proof}

\begin{cor}\label{sansucthm-cor2}
Sous les hypoth\`eses du th\'eor\`eme \ref{sansucthm}, supposons que $H^2(k,H^1(Z_{\bk},\mu_{\infty}))=0$ et $H^2(Z_{\bk},\mu_{\infty})=0$. 
Si $H^3(k,\mu_{\infty})=0$ ou $X(k)\neq\emptyset$, alors les homomorphismes de Sansuc
 $\lambda: H^2_G(X,\mu_{\infty})\to H^2_{1,e}(G,\mu_{\infty})$ et $\lambda_{\Br}: \Br'_G(X)\to \Br_{2/3,e}(G)$ sont surjectifs.
\end{cor}

\begin{proof}
Par hypoth\`eses, le th\'eor\`eme \ref{sansucthm} induit une suite exacte:
$$0\to H^1(Z_{\bk},\mu_{\infty})\to H^1(X_{\bk},\mu_{\infty})\xrightarrow{\lambda_{H^1}} H^1(G_{\bk},\mu_{\infty})\to 0$$
et donc un homomorphisme surjectif $$\lambda_{H^1,*}: H^1(k, H^1(X_{\bk},\mu_{\infty}))\to H^1(k,H^1(G_{\bk},\mu_{\infty})).$$
Puisque $H^3(k,\mu_{\infty})=0$ ou $X(k)\neq\emptyset$, d'apr\`es (\ref{prop21.1e5}), $H^2_1(X,\mu_{\infty})\to H^1(k, H^1(X_{\bk},\mu_{\infty}))$ est surjectif.
D'apr\`es (\ref{prop21.1e4}), $\lambda$ est surjectif et donc $\lambda_{\Br}$ est surjectif.
\end{proof}

\begin{cor}\label{sansucthm-cor3}
Soient $G$ un groupe alg\'ebrique connexe et $(X,\rho)$ un $G$-torseur. 
Si $H^3(k,\mu_{\infty})=0$ ou $X(k)\neq\emptyset$, alors on a des suites exactes:
$$H^2(k,\mu_{\infty})\to H^2_G(X,\mu_{\infty})\xrightarrow{\lambda} H^2_{1,e}(G,\mu_{\infty})\to 0 \ \ \ \text{et}\ \ \ \Br(k)\to \Br'_G(X)\xrightarrow{\lambda_{\Br}}\Br_{2/3,e}(G)\to 0.$$
\end{cor}

\begin{proof}
D'apr\`es le corollaire \ref{sansucthm-cor2}, $\lambda$ et $\lambda_{\Br}$ sont surjectifs. 
Le corollaire \ref{sansucthm-cor} donne la premi\`ere suite exacte. 
D'apr\`es la proposition \ref{prop21.2} (2), $\Br'_G(X)=\Br_{2/3}(X)$.
Donc il suffit de montrer que $ \Br_{2/3}(X)/\Im \Br(k)\to \Br_{2/3,e}(G)$ est injectif.

Si $X(k)\neq\emptyset$, ceci d\'ecoule de la proposition \ref{prop21.2} (4).
En g\'en\'eral, soit $K:=k(X)$.
 Alors $X_K(K)\neq\emptyset$ et donc $\Br_{2/3}(X_K)/\Im \Br(K)\to \Br_{2/3,e}(G_K)$ est injectif.
D'apr\`es le corollaire \ref{lem21.2cor1}, $\Br_{2/3}(X)/\Im \Br(k)\to \Br_{2/3}(X_K)/\Im \Br(K)$ et $\Br_{2/3,e}(G)\to \Br_{2/3,e}(G_K)$ sont injectifs, d'o\`u le r\'esultat.
\end{proof}

\medskip

Les deux propositions suivantes d\'ecrivent les propri\'et\'es de $H^2_G(X,\mu_{\infty})$ et de $\Br'_G(X)$ par rapport au changement de $X$ et de $G$.

\begin{prop}\label{sansucthm-cor1}
Soient $G$, $H$ deux groupes alg\'ebriques connexes et $p: Y\to X$ un $G$-morphisme de $G$-vari\'et\'es. 
Si $p$ est un $H$-torseur, alors 

(1) on a $(p^*)^{-1}H^2_G(Y,\mu_{\infty})=H^2_G(X,\mu_{\infty})$, o\`u $p^*: H^2(X,\mu_{\infty})\to H^2(Y,\mu_{\infty})$ est l'homomorphisme induit par $p$;

(2) si $H$ est lin\'eaire, on a $(p^*_{\Br})^{-1}\Br'_G(Y)=\Br'_G(X)$, o\`u $\Br(X)\xrightarrow{p^*_{\Br}}\Br(Y)$.
\end{prop}

\begin{proof}
Notons $\rho: G\times X\to X$, $\rho_Y:G\times Y\to Y$ deux actions et $p_2:G\times X\to X$, $p_1: G\times X\to G$, $p_{2,Y}: G\times Y\to Y$ les projections.

Puisque $G\times Y\xrightarrow{id_G\times p}G\times X$ est aussi un $H$-torseur, d'apr\`es le th\'eor\`eme \ref{sansucthm}, on a un diagramme commutatif de suites exactes
$$\xymatrix{H^1_e(H,\mu_{\infty})\ar[r]\ar[d]^=&H^2(X,\mu_{\infty})\ar[r]^{p^*}\ar[d]^{\phi^*}&H^2(Y,\mu_{\infty})\ar[d]^{\phi_{Y}^*}\\
H^1_e(H,\mu_{\infty})\ar[r]&H^2(G\times X,\mu_{\infty})\ar[r]^-{(id_G\times p)^*}&H^2(G\times Y,\mu_{\infty}),
}$$
o\`u $\phi=\rho$ (resp. $\phi=p_2$) et $\phi_Y= \rho_Y$ (resp. $\phi_Y=p_{2,Y}$). 
 Donc, pour tout $\alpha\in (p^*)^{-1}H^2_G(Y,\mu_{\infty})$, on a
$$\rho^*(\alpha)-p_2^*(\alpha)\in p_1^*H^2(G,\mu_{\infty})+\Im H^1_e(H,\mu_{\infty})\sbt p_1^*H^2(G,\mu_{\infty})+p_2^*H^2(X,\mu_{\infty}).$$ 
Puisque $(\rho^*(\alpha)-p_2^*(\alpha))|_{e_G\times X}=0 $, on a $\rho^*(\alpha)-p_2^*(\alpha)\in p_1^*H^2(G,\mu_{\infty})$. Ceci donne (1).

Pour (2), par la suite exacte de Sansuc \cite[Prop. 6.10]{S}, on a une suite exacte $$\Pic(X)\to \Pic(Y)\to \Pic(H)$$ avec $\Pic(H)$ fini.
Donc $\Pic(X)\otimes \BQ/\BZ\to \Pic(Y)\otimes \BQ/\BZ$ est surjectif.
Par la suite exacte de Kummer, on a un diagramme commutatif de suites exactes:
$$\xymatrix{0\ar[r]&\Pic(X)\otimes \BQ/\BZ \ar[r]\ar[d]^{p^*_{\Pic}}& H^2(X,\mu_{\infty})\ar[r]\ar[d]^{p^*}& \Br(X)\ar[r]\ar[d]^{p^*_{\Br}}&0\\
0\ar[r]&\Pic(Y)\otimes \BQ/\BZ \ar[r]& H^2(Y,\mu_{\infty})\ar[r]& \Br(Y)\ar[r]&0.
}$$
Puisque $p^*_{\Pic}$ est surjectif, une chasse au diagramme donne (2).
\end{proof}

Dans le cas o\`u $X\cong G/H$ avec $H\subset G$ un sous-groupe connexe, $p: G\to X$ est un $H$-torseur. 
La proposition \ref{sansucthm-cor1} (1) et la proposition \ref{prop21.1} (2) impliquent
 $$H^2_G(X,\mu_{\infty})\subset \ker(H^2(X,\mu_{\infty})\xrightarrow{p^*} H^2(G_{\bk},\mu_{\infty})) $$
et donc
\begin{equation}\label{sansucthm-cor1-e1}
 \Br'_G(X) \subset \Br_1(X,G):=\ker(\Br(X)\xrightarrow{p^*}  \Br(G_{\bk})) .
 \end{equation}

\begin{prop}\label{sansucthm-prop}
Soient $1\to N\to H\xrightarrow{\psi}G\to 1$ une suite exacte de groupes alg\'ebriques connexes, et
$(X,\rho )$ une $G$-vari\'et\'e lisse g\'eom\'e\-triquement int\`egre.
Alors $H^2_G(X,\mu_{\infty})=H^2_H(X,\mu_{\infty})  $ et $\Br'_G(X)=\Br'_H(X)$.
\end{prop}

\begin{proof}
Puisque $H\times X\xrightarrow{\psi_X} G\times X$ est un $N$-torseur,
d'apr\`es le corollaire \ref{sansucthm-cor},  on a un diagramme commutatif de suites exactes:
$$\xymatrix{H^1_e(N,\mu_{\infty})\ar[d]^=\ar[r]&H^2(G,\mu_{\infty})\ar[r]\ar[d]^{p_1^*}&H^2_N(H,\mu_{\infty})\ar[r]\ar[d]^{p_1^*}&H^2_{1,e}(N,\mu_{\infty})\ar[d]^=\\
H^1_e(N,\mu_{\infty})\ar[r]&H^2(G\times X,\mu_{\infty}))\ar[r]^{\psi_X^*}&H^2_N(H\times X,\mu_{\infty}))\ar[r]&H^2_{1,e}(N,\mu_{\infty}).
}$$
 Donc $ (\psi_X^*)^{-1}(p_1^*H^2_N(H,\mu_{\infty}))=p_1^*H^2(G,\mu_{\infty})$.
 Puisque $H^2_{1,e}(H,\mu_{\infty}) \sbt H^2_1(H,\mu_{\infty}) \sbt H^2_N(H,\mu_{\infty}) $ (Proposition \ref{prop21.1} (1)), 
 on a $ (\psi_X^*)^{-1}(p_1^*H^2_{1,e}(H,\mu_{\infty}) ) \subset p_1^*H^2(G,\mu_{\infty}) $.
 La proposition \ref{prop21.1} (3) donne
 $$H^2_H(X,\mu_{\infty})=\{a\in H^2(X,\mu_{\infty})| \psi_X^*(( \rho^*-p_2^* )(a)) \in  p_1^*H^2_{1,e}(H,\mu_{\infty}) \} $$  
 $$\subset   \{a\in H^2(X,\mu_{\infty})| ( \rho^*-p_2^* )(a) \in p_1^*H^2(G,\mu_{\infty})  \} = H^2_G(X,\mu_{\infty}),$$
 d'o\`u le r\'esultat.
\end{proof}

\begin{cor}\label{corbraueralgebraic}
Sous les hypoth\`eses de la proposition \ref{sansucthm-prop}, supposons qu'il existe une $H$-vari\'et\'e $Y$  et un $H$-morphisme $Y\xrightarrow{p}X$ tels que $Y\to X $ soit un $N$-torseur.
Alors 

(1) $H^2(X,\mu_{\infty})\xrightarrow{p^*} H^2(Y,\mu_{\infty})$ satisfait $(p^*)^{-1} H^2_H(Y,\mu_{\infty})=H^2_G(X,\mu_{\infty})$, et on a une suite exacte (o\`u $\lambda$ est l'homomorphisme de Sansuc), fonctorielle en $(X,Y,p,N)$:
$$0\to H^1(X,\mu_{\infty})\to H^1(Y,\mu_{\infty})\to  H^1_e(N,\mu_{\infty})\xrightarrow{\chi} H^2_G(X,\mu_{\infty})\xrightarrow{p^*} H^2_H(Y,\mu_{\infty})\xrightarrow{\lambda}H^2_{1,e}(N,\mu_{\infty}).$$

(2) si $N$ est lin\'eaire, $\Br(X)\xrightarrow{p^*}\Br(Y)$ satisfait $(p^*)^{-1}\Br'_H(Y)=\Br'_G(X)$ et on a  une suite exacte (o\`u $\lambda_{\Br}$ est l'homomorphisme de Sansuc), fonctorielle en $(X,Y,p,N)$:
$$\Pic(Y)\to  \Pic(N)\xrightarrow{\chi} \Br'_G(X)\xrightarrow{p^*} \Br'_H(Y)\xrightarrow{\lambda_{\Br}}\Br_e(N).$$
\end{cor}

\begin{proof}
Une application du th\'eor\`eme \ref{sansucthm} et du corollaire \ref{sansucthm-cor} au $N$-torseur $Y\to X $ donne une suite exacte:
$$0\to H^1(X,\mu_{\infty})\to H^1(Y,\mu_{\infty})\to  H^1_e(N,\mu_{\infty})\xrightarrow{\chi} H^2(X,\mu_{\infty})\xrightarrow{p^*} H^2_N(Y,\mu_{\infty})\xrightarrow{\lambda}H^2_{1,e}(N,\mu_{\infty}).$$
Une application de la proposition \ref{sansucthm-cor1} au $N$-torseur $p: Y\to X $ (avec l'action de $H$) donne
  (en utilisant la proposition \ref{sansucthm-prop})
$$\chi (H^1_e(N,\mu_{\infty}))\sbt (p^*)^{-1}(0)\sbt  (p^*)^{-1}(H^2_H(Y,\mu_{\infty}))=H^2_H(X,\mu_{\infty})=H^2_G(X,\mu_{\infty}),$$
d'o\`u on obtient (1).

L'\'enonc\'e (2) d\'ecoule du m\^eme argument que (1) en rempla\c{c}ant le corollaire \ref{sansucthm-cor} par \cite[Thm. 3.10]{C1}.
\end{proof}

\section{Sp\'ecialisation du sous-groupe invariant} \label{5}

Dans toute cette section,  $k$ est un corps quelconque de caract\'eristique $0$. 
Sauf  mention explicite du contraire,  une vari\'et\'e est une $k$-vari\'et\'e.

Consid\'erons maintenant une suite exacte de groupes alg\'ebriques connexes 
\begin{equation}\label{Dlembrauersurjprop-hyp-e}
1\to G \xrightarrow{\varphi } H \xrightarrow{\psi} T \to 1
\end{equation}
 avec $T$ un tore. 
Pour une $H$-vari\'et\'e $Y$ munie d'un $H$-morphisme $f: Y\to T$ et pour tout $t\in T(k)$, la fibre $Y_t$ est une $G$-vari\'et\'e (voir le th\'eor\`eme \ref{proppropbrauersuj} ci-dessous).
Dans \cite[Prop. 3.13]{C1},  on calcule le noyau et le conoyau de l'homomorphisme de sp\'ecialisation $\Br_H(Y)\to \Br_G(Y_t)$ lorsque $H$ est lin\'eaire.
Dans cette section, on g\'en\'eralise ce r\'esultat au cas o\`u $H$  n'est pas n\'ecessairement lin\'eaire (le th\'eor\`eme \ref{proppropbrauersuj}).

 \medskip
 
 La proposition \ref{Dlembrauersurjprop}  suivante g\'en\'eralise \cite[Lem. 5.5]{C1}.

\begin{prop}\label{Dlembrauersurjprop}
Pour la suite exacte (\ref{Dlembrauersurjprop-hyp-e}), on a un diagramme commutatif de suites exactes:
$$\xymatrix{ H^1_e(G,\mu_{\infty})\ar[r] \ar[d]^=& H^2_{1,e}(T,\mu_{\infty})\ar[r]^{\psi^*}\ar@{^{(}->}[d]& H^2_{1,e}(H,\mu_{\infty})\ar[r]^{\varphi^*} \ar@{^{(}->}[d]\ar@{}[rd]|{(1)}& H^2_{1,e}(G,\mu_{\infty})\ar[r]\ar[d]^=& H^3(k,T^*)\\
 H^1_e(G,\mu_{\infty})\ar[r]& H^2_T(T,\mu_{\infty})\ar[r]^{\psi^*}&H^2_H(H,\mu_{\infty})\ar[r]^{\lambda}&H^2_{1,e}(G,\mu_{\infty}) &,
}$$
o\`u la deuxi\`eme ligne est induite par application du corollaire \ref{corbraueralgebraic} (1) au $G$-torseur $H\to T$.
\end{prop}

\begin{proof}
On applique le corollaire \ref{corbraueralgebraic} (1) au  $G$-torseur $H\to T$ et on obtient la suite exacte de la deuxi\`eme ligne. 

D'apr\`es la proposition \ref{prop21.1},  $e_T^*: H^2_1(T,\mu_{\infty})\to H^2(k,\mu_{\infty})$ \'equivaut \`a la composition:
$$H^2_1(T,\mu_{\infty})\cong H^2_T(T,\mu_{\infty}) \to H^2_H(H,\mu_{\infty})\cong H^2_1(H,\mu_{\infty}) \xrightarrow{e_H^*} H^2(k,\mu_{\infty}).$$
Donc $\Im  H^1_e(G,\mu_{\infty})\sbt \Ker(e_T^*)$ et la premi\`ere ligne est exacte en $H^2_{1,e}(T,\mu_{\infty})$.

Par la suite exacte de Kummer, $H^1(T_{\bk},\mu_{\infty})\cong T^*\otimes \BQ/\BZ$ et donc 
\begin{equation}\label{Dlembrauersurjprop-e1}
H^2(k,H^1(T_{\bk},\mu_{\infty}))\cong H^3(k,T^*).
\end{equation}
Le corollaire \ref{corCartan1} (ii) induit une suite exacte 
$$0\to H^1(T_{\bk},\mu_{\infty})\to H^1(H_{\bk},\mu_{\infty})\xrightarrow{\varphi^*} H^1(G_{\bk},\mu_{\infty})\to 0.$$
 Appliquons $H^i(k,-)$ \`a cette suite exacte. Le corollaire \ref{prop21.2lem} et (\ref{Dlembrauersurjprop-e1}) donnent une suite exacte:
 $$ H^2_{1,e}(T,\mu_{\infty})\to H^2_{1,e}(H,\mu_{\infty})\xrightarrow{\varphi^*} H^2_{1,e}(G,\mu_{\infty})\to H^3(k,T^*).$$
Ceci donne la premi\`ere suite exacte.

D'apr\`es la proposition \ref{prop21.1} (4), le carr\'e (1) est commutatif et donc le diagramme est commutatif.
\end{proof}

\begin{cor}\label{Dlembrauersurj}
Dans la suite exacte (\ref{Dlembrauersurjprop-hyp-e}), 
si l'on a $H^3(k,T^*)=0$, alors les homomorphismes $H^2_{1,e}(H,\mu_{\infty})\to H^2_{1,e}(G,\mu_{\infty})$ et $\Br_{2/3,e}(H)\xrightarrow{\varphi^*} \Br_{2/3,e}(G)$ sont surjectifs.
\end{cor}

\begin{proof}
Ceci d\'ecoule de la proposition \ref{Dlembrauersurjprop} et du corollaire \ref{prop21.2lem}.
\end{proof}

\begin{lem}\label{Dlemsuite1}
Soient $G$, $N$ deux groupes alg\'ebriques connexes et
$X$ une $G$-vari\'et\'e lisse g\'eom\'e\-triquement int\`egre.  
Soient $H:=N\times G$ et $P$ une $H$-vari\'et\'e tels que $P$ soit un $N$-torseur sur $k$.
Soient $Y:= P\times X$ et $Y\xrightarrow{p_1}P$, $Y\xrightarrow{p_2}X$ les deux projections. Supposons donn\'ee une action:
 $$H\times Y\to Y:\ (n,g)\times (p,x)\mapsto ((n,g)\cdot p,g\cdot x).$$
Alors 

(1) si $P=N$, on a un isomorphisme:
$(p_1^*,p_2^*): H^2_{1,e}(N,\mu_{\infty})\oplus H^2_G(X,\mu_{\infty}) \iso H^2_H(Y,\mu_{\infty});  $
 
(2) si $H^3(k,\mu_{\infty})=0$,  on a un isomorphisme:
$$(p_1^*,p_2^*): H^2_1(P,\mu_{\infty})/\Im H^2(k,\mu_{\infty})\oplus H^2_G(X,\mu_{\infty})/\Im H^2(k,\mu_{\infty}) \iso H^2_H(Y,\mu_{\infty})/\Im H^2(k,\mu_{\infty}).  $$
\end{lem}

\begin{proof}
Par la proposition \ref{prop21.1} (2), 
on a $H^2_N(P,\mu_{\infty})\cong H^2_1(P,\mu_{\infty})$.
D'apr\`es le corollaire \ref{corbraueralgebraic} (1) et l'isomorphisme $H^2_{1,e}(N,\mu_{\infty})\cong H^2_1(N,\mu_{\infty})/\Im H^2(k,\mu_{\infty})$, on a un diagramme commutatif de suites exactes
$$\xymatrix{H^1(P,\mu_{\infty})\ar[r]^{\vartheta_1}\ar[d]&H^1_e(N,\mu_{\infty})\ar[d]^=\ar[r]&H^2(k,\mu_{\infty})\ar[r]\ar[d] 
&H^2_1(P,\mu_{\infty})\ar[d]^{p_1^*}\ar[r]^{\vartheta_2}&H^2_{1,e}(N,\mu_{\infty})\ar[d]^=\\
H^1(Y,\mu_{\infty})\ar[r]&H^1_e(N,\mu_{\infty}) \ar[r]&H^2_G(X,\mu_{\infty})\ar[r]^-{p_2^*}&
 H^2_H(Y,\mu_{\infty})\ar[r] &H^2_{1,e}(N,\mu_{\infty}).
}$$
Puisque $P(k)\neq\emptyset$ ou $H^3(k,\mu_{\infty})=0$, d'apr\`es le corollaire \ref{sansucthm-cor3}, l'homomorphisme   $\vartheta_2$ est surjectif.
Une chasse au diagramme donne une suite exacte 
\begin{equation}\label{Dlemsuite1-e}
 H^2(k,\mu_{\infty})\xrightarrow{\vartheta_3} H^2_1(P,\mu_{\infty})\oplus H^2_G(X,\mu_{\infty}) \xrightarrow{(p_1^*,p_2^*)} H^2_H(Y,\mu_{\infty})\to 0,
\end{equation}
et, si $P(k)\neq\emptyset$, l'homomorphisme $\vartheta_3$ est injectif, d'o\`u le r\'esultat.
\end{proof}

\begin{thm}\label{proppropbrauersuj}
Consid\'erons la suite exacte (\ref{Dlembrauersurjprop-hyp-e}).
Soient $Y$ une $H$-vari\'et\'e lisse, g\'eom\'e\-triquement int\`egre et $Y\xrightarrow{f}T$ un $H$-morphisme.
Notons $H^2_{1,e}(H,\mu_{\infty})\xrightarrow{\varphi^*}H^2_{1,e}(G,\mu_{\infty})$ l'homomorphisme induit par $\varphi: G\to H$.
Alors, pour tout $t\in T(k)$, la fibre $Y_t$ est $G$-invariante et 
on a une suite exacte naturelle
$$H^2_{1,e}(T,\mu_{\infty})\to H^2_H(Y,\mu_{\infty})\to H^2_{G}(Y_t,\mu_{\infty})\to \coker (\varphi^*).$$
\end{thm}

\begin{proof}
D'apr\`es \cite[Prop. 2.2]{CX1}, $Y_t$ est lisse, g\'eom\'etriquement int\`egre.
Puisque $T$ est commutatif, la fibre $Y_t$ est $G$-invariante.
Notons:
 $$Y_t\xrightarrow{i}H\times Y_t: y\mapsto (e_H,y)\ \ \ \text{et}\ \ \  H\times Y_t\xrightarrow{\rho}Y: (h,y)\mapsto h\cdot y .$$
Alors $\rho\circ i $ est l'immersion $Y_t\sbt Y$. On fixe des actions
$$H\times G\curvearrowright H\times Y_t: (h,g)\times (h',y)\mapsto (hh'g^{-1},g\cdot y)\ \ \ \text{et}\ \ \ 
H\times G\curvearrowright H: (h,g)\times h'\mapsto hh'g^{-1}. $$
Par d\'efinition, $Y\cong H\times^{G}Y_t$ est le produit contract\'e (cf. \cite[Lem. 2.2.3]{sko}) et on a un diagramme commutatif de $H\times G$-morphismes
$$\xymatrix{Y_t&H\times Y_t\ar[l]^-{p_2}\ar[r]_-{p_1}\ar[d]^{\rho}&H\ar[d]^{t\cdot \psi (-)}\\
&Y\ar[r]^f&T
}$$
tels que les colonnes soient  des $G$-torseurs. 

On applique le lemme \ref{Dlemsuite1} (1) \`a la $H\times G$-vari\'et\'e $H\times Y_t$, et on obtient un isomorphisme
 \begin{equation}\label{proppropbrauersuj-e1}
 H^2_{1,e}(H,\mu_{\infty})\oplus H^2_{G}(Y_t,\mu_{\infty})\xrightarrow{(p_1^*,p_2^*)}H^2_{H\times G}(H\times Y_t,\mu_{\infty}) .
 \end{equation}
Le corollaire \ref{corbraueralgebraic} (1) et la proposition \ref{Dlembrauersurjprop} donnent un diagramme commutatif de suites exactes:
$$\xymatrix{H^1_e(G,\mu_{\infty})\ar[r]\ar[d]^=&H^2_{1,e}(T,\mu_{\infty})\ar[r]\ar[d]&H^2_{1,e}(H,\mu_{\infty})\ar[r]^{\varphi^*}\ar[d]^{p_1^*}&H^2_{1,e}(G,\mu_{\infty})\ar[d]^=\\
H^1_e(G,\mu_{\infty})\ar[r]&H^2_H(Y,\mu_{\infty})\ar[r]^-{\rho^*}&H^2_{H\times G}(H\times Y_t,\mu_{\infty}) \ar[r] \ar[d]^{i^*}&H^2_{1,e}(G,\mu_{\infty})\\
&&H^2(Y_t,\mu_{\infty})&.
}$$
Puisque $ p_2\circ i=id$ et $i^*\circ p_1^*=0$, d'apr\`es (\ref{proppropbrauersuj-e1}), on a $H^2_{G}(Y_t,\mu_{\infty})= \Im (i^*)\cong \coker(p_1^*)$.
Une chasse au diagramme donne l'\'enonc\'e.
\end{proof}

\begin{cor}\label{propbrauersuj}
Sous les hypoth\`eses de le th\'eor\`eme \ref{proppropbrauersuj},
supposons $H^3(k,T^*)=0$.
Alors, pour tout $t\in T(k)$, les homomorphismes $H^2_H(Y,\mu_{\infty})\to H^2_{G}(Y_t,\mu_{\infty})$ et $ \Br'_H(Y)\to \Br'_{G}(Y_t)$ sont surjectifs.
\end{cor}

\begin{proof}
Ceci d\'ecoule du th\'eor\`eme \ref{proppropbrauersuj} et de la proposition \ref{Dlembrauersurjprop}.
\end{proof}

\section{Descente}

Dans toute cette section,  $k$ est un corps de nombres. 
Sauf  mention explicite du contraire,  une vari\'et\'e est une $k$-vari\'et\'e.

La m\'ethode de descente des points ad\'eliques est \'etablie par Colliot-Th\'el\`ene et Sansuc dans \cite{CTS}.
L'auteur \'etudie la m\'ethode de descente des points ad\'eliques orthogonaux aux sous-groupes de Brauer invariants dans \cite{C1,C5} et, en particulier, \cite[Prop. 5.7]{C5} est utilis\'e dans la d\'emonstration de \cite[Thm. 1.4]{C5}.

Dans cette section, on donne une variante de \cite[Prop. 5.1]{C5} et une variante de \cite[Prop. 5.7]{C5} dans le cas o\`u le $k$-groupe n'est pas forcement lin\'eaire.
La premi\`ere est la proposition \ref{descprop4.1}, qui \'etablit la formule de descente par rapport au sous-groupe de Brauer invariant pour un torseur sous un $k$-groupe de type multiplicatif. 
La deuxi\`eme est la proposition \ref{prop21.3}, qui suit l'id\'ee de Demarche dans  \cite[Prop. 5]{D09}.

\begin{lem}\label{descprop4.1lem1}
Soient $1\to T\to H\xrightarrow{\psi}G\to 1$ une suite exacte de groupes alg\'ebriques connexes avec $T$ un tore quasi-trivial, 
$X$ une $G$-vari\'et\'e lisse g\'eom\'e\-triquement int\`egre, $Y$ une $H$-vari\'et\'e lisse  et $Y\xrightarrow{f}X$ un $H$-morphisme tels que $Y\to X $ soit un $T$-torseur.
Alors on a  $$X(\RA_k)^{\Br'_G(X)}=f(Y(\RA_k)^{\Br'_{H}(Y)}).$$
\end{lem}

\begin{proof}
Ceci d\'ecoule de \cite[(5.2)]{C5} et du corollaire \ref{corbraueralgebraic} (2).
\end{proof}

\begin{prop}\label{descprop4.1}
Soient $G$, $H$ deux groupes alg\'ebriques connexes et $\psi: H\to G$ un homomorphisme surjectif de noyau central $S$ de type multiplicatif.
Soient $X$ (resp. $Y$) une $G$-vari\'et\'e (resp. $H$-vari\'et\'e) lisse g\'eom\'etriquement int\`egre et $f: Y\to X$ un $H$-morphisme tels que $Y$ soit un $S$-torseur sur $X$, o\`u l'action de $S$ est induite par l'action de $H$.
Alors, pour tout $\sigma\in H^1(k,S)$, le tordu $Y_{\sigma}$ est une $H$-vari\'et\'e et on a:
$$X(\RA_k)^{\Br'_G(X)}=\cup_{\sigma\in H^1(k,S)}f_{\sigma}(Y_{\sigma}(\RA_k)^{\Br'_H(Y_{\sigma})}).$$
\end{prop}

\begin{proof}
On suit la d\'emonstration de \cite[Prop. 5.1]{C5}.

D'apr\`es \cite[Prop. 1.3]{CTS1} et \cite[Lem. 5.4]{C1}, il existe une suite exacte
\begin{equation}\label{BiDprop4.1e1}
0\to S\to T_0\xrightarrow{\varphi} T\to 0
\end{equation}
o\`u  $T_0$  est un tore quasi-trivial et   le groupe des caract\`eres $T^*$ du tore $T$ v\'erifie $H^3(k,T^*)=0$.

Soit $H_0:=H\times^ST_0$ le produit contract\'e (cf. \cite[Lem. 2.2.3]{sko}). 
Alors $H_0$ est un groupe alg\'ebrique connexe et $H \xrightarrow{\psi} G$ induit une suite exacte 
$$1\to T_0\to H_0\xrightarrow{\psi_0} G\to 1.$$
Soit $Y_0:=Y\times^ST_0$. Notons $i: Y\to Y_0$ l'immersion ferm\'ee canonique. 
Alors $ Y_0$ est une $H_0$-vari\'et\'e et $f$ induit un $H_0$-morphisme $Y_0\xrightarrow{f_0} X$ 
tels que $f_0$ est un $T_0$-torseur.
D'apr\`es le lemme \ref{descprop4.1lem1}, on a 
$$X(\RA_k)^{\Br'_G(X)}=f_0(Y_0(\RA_k)^{\Br'_{H_0}(Y_0)}).$$

L'isomorphisme $Y_0\times^{T_0}T\cong Y\times^ST_0\times^{T_0}T\cong X\times T$ 
induit un $T_0$-morphisme $\phi: Y_0\to T$ tel que $\phi^{-1}(e_T)=i(Y)$.
D'apr\`es des arguments classiques (voir la d\'emonstration de \cite[Thm. 5.9]{C1}), 
pour tout $t\in T(k)$, on a $\phi^{-1}(t)\cong Y_{\partial(t)}$ 
et le morphisme $\phi^{-1}(t)\hookrightarrow Y_0\xrightarrow{f_0}X $ est exactement $f_{\partial(t)}$, 
o\`u $\partial: T(k)\twoheadrightarrow H^1(k,S)$ est l'homomorphisme induit par (\ref{BiDprop4.1e1}).
D'apr\`es le th\'eor\`eme \ref{proppropbrauersuj} et le corollaire \ref{propbrauersuj}, $\phi^{-1}(t)$ est une $H$-vari\'et\'e et l'homomorphisme canonique
$\Br'_{H_0}(Y_0)\to \Br'_H(\phi^{-1}(t))$ est surjectif pour tout $t\in T(k)$.

D'apr\`es \cite[Thm. 5.1]{CLX}, on a 
$$T(\RA_k)^{\Br_1(T)}=\varphi(T_0(\RA_k)^{\Br_1(T_0)})\cdot T(k).$$
D'apr\`es le corollaire \ref{corlinBr1et2/3}, on a 
$$\Br_1(T)\cong \Br'_T(T),\ \ \ \Br_1(T_0)\cong \Br'_{T_0}(T_0)\ \ \ 
\text{et}\ \ \  \phi(Y_0(\RA_k)^{\Br'_{H_0}(Y_0)})\sbt T(\RA_k)^{\Br_1(T)}.$$
D'apr\`es (\ref{invequa2.3}), $Y_0(\RA_k)^{\Br'_{H_0}(Y_0)}$ est $T_0(\RA_k)^{\Br_1(T_0)}$-invariant. 
Ceci implique:
$$Y_0(\RA_k)^{\Br'_{H_0}(Y_0)}=T_0(\RA_k)^{\Br_1(T_0)}\cdot (\sqcup_{t\in T(k)}\phi^{-1}(t)(\RA_k)^{\Br'_H(\phi^{-1}(t))}), $$
et donc $X(\RA_k)^{\Br'_G(X)}=f_0 [\sqcup_{t\in T(k)}\phi^{-1}(t)(\RA_k)^{\Br'_H(\phi^{-1}(t))}]
=\cup_{t\in T(k)}f_{\partial(t)}[Y_{\partial(t)}(\RA_k)^{\Br'_H(Y_{\partial(t)})}] .$
\end{proof}

\bigskip

Pour toute vari\'et\'e lisse $X$, d\'efinissons $X(\RA_k^{nc})$ l'espace des points ad\'eliques de $X$ hors des places complexes,
i.e. on a 
$X(\RA_k)\cong (\prod_{v\ \text{complexe}}X(k_v))\times X(\RA_k^{nc}).$
De plus, on a:
\begin{equation}\label{lem4.22.1-e}
X(\RA_k)^{ob}\cong (\prod_{v\ \text{complexe}}X(k_v))\times X(\RA_k^{nc})^{ob}
\end{equation}
pour l'obstruction $ob=\Br(X) $ ou  $ob=\Br_{2/3}(X)$ ou $ob=\et, \Br$ ou, si $X$ est une $G$-vari\'et\'e pour un groupe alg\'ebrique connexe $G$, pour $ob=\Br'_G(X)$ ou $ob=G-\et,\Br'_G$.

Le lemme suivant g\'en\'eralise \cite[Lem. 5.2]{C5} (voir \cite[Lem. 2.2.8]{Dth} pour une variante).

\begin{lem}\label{lem4.22.1}
Soient $X$ une vari\'et\'e lisse et $\{X_i\}_{i\in I}$ les composantes connexes de $X$ 
telles que $X_i$ soit g\'eom\'etriquement int\`egre pour tout $i\in I$.
Alors on a:
$$X(\RA_k^{nc})^{\Br_{2/3}(X)}=\coprod_{i\in I}  X_i(\RA_k^{nc})^{\Br_{2/3}(X_i)}$$
et, si $X$ est une $G$-vari\'et\'e pour un groupe lin\'eaire connexe $G$, on a:
$$X(\RA_k^{nc})^{\Br'_G(X)}=\coprod_{i\in I}  X_i(\RA_k^{nc})^{\Br'_G(X_i)} \ \ \ \text{et} \ \ \  
X(\RA_k^{nc})^{G-\et,\Br'_G}=\coprod_{i\in I}  X_i(\RA_k^{nc})^{G-\et, \Br'_G} .$$
\end{lem}

\begin{proof}
Puisque $\Br_{2/3}(-)$ (resp. $\Br_G'(-)$, resp. l'ensemble des $F$-torseurs, resp. l'ensemble des $F$-torseurs $G$-compatibles pour un $k$-groupe fini $F$) de $X$
est la somme directe de celui des composantes connexes de $X$, on obtient l'inclusion $\supset $ dans les trois cas ci-dessus.

Par ailleurs, soit $\pi_0(X)$ le sch\'ema des composantes connexes g\'eom\'etriques de $X$, 
i.e. $\pi_0(X)$ est un $k$-sch\'ema fini \'etale
et il existe un $k$-morphisme surjectif $\phi: X\to \pi_0(X)$ de fibres g\'eom\'e\-triquement int\`egres.
D'apr\`es \cite[Prop. 3.3]{LX}, on a $\pi_0(X)(\RA_k^{nc})^{\Br(\pi_0(X))}=\pi_0(X)(k)$.
Par d\'efinition, $\phi^*(\Br(\pi_0(X)))\sbt \Br_{2/3}(X)$ et $\phi^*(\Br(\pi_0(X)))\sbt \Br'_G(X)$, d'o\`u l'on obtient l'inclusion ~$\sbt$.
\end{proof}

\begin{prop}\label{prop21.3}
Soit $X$ une $k$-vari\'et\'e lisse g\'eom\'e\-triquement int\`egre.  Soit 
$$1\to S \to L \xrightarrow{\psi} F\to 1$$
une suite exacte de $k$-groupes finis.
Soient $ V\to X$ un $L$-torseur et $Y:=V/S \to X$ le $F$-torseur induit par $\psi$.
Supposons que $S$ est contenu dans le centre de $L$.
Alors, pour tout $\sigma\in H^1(k,F)$ avec $Y_{\sigma}(\RA_k)^{\Br_{2/3}(Y_{\sigma})}\neq\emptyset$,
il existe un $\alpha\in H^1(k,L)$ tel que $\psi_*(\alpha)=\sigma$.
\end{prop}

\begin{proof}
D'apr\`es le lemme \ref{lem4.22.1}, il existe une composante connexe $X'\sbt Y_{\sigma}$ telle que $X'(\RA_k)^{\Br_{2/3}(X')}\neq\emptyset$.
Ainsi, $X'$ est g\'eom\'e\-triquement int\`egre sur $k$.

D'apr\`es \cite[Prop. 1.3]{CTS1}, il existe une suite exacte 
$0\to S\to T\to T_0\to 0$ avec $T$ un tore flasque et $T_0$ un tore quasi-trivial.
Soit $L':=L\times^ST $ le produit contract\'e (cf. \cite[Lem. 2.2.3]{sko}). 
Alors $L'$ est un groupe lin\'eaire, car  $S$ est contenu dans le centre de $L$.
Ceci induit un diagramme commutatif de suites exactes et de colonnes exactes:
$$\xymatrix{&1\ar[d]&1\ar[d]&&\\
1\ar[r]&S\ar[r]\ar[d]^{\phi}&L\ar[r]^{\psi}\ar[d]^{\psi_2}&F\ar[d]^=\ar[r]&1\\
1\ar[r]&T\ar[r]\ar[d]&L'\ar[r]^{\psi_1}\ar[d]&F\ar[r]&1\\
&T_0\ar[r]^=\ar[d]&T_0\ar[d]&&\\
&0&0&&.
}$$
Puisque $H^1(k,T_0)=0$, ceci induit un diagramme commutatif de suites exactes d'ensembles point\'es:
$$ \xymatrix{H^1(k,L)\ar[r]^{\psi_*}\ar@{->>}[d]&H^1(k,F)\ar[d]^=\ar[r]&H^2(k,S)\ar[d]\\
H^1(k,L')\ar[r]^{\psi_{1,*}}&H^1(k,F)\ar[r]&H^2(k,T).
} $$
D'apr\`es \cite[(2)]{HS13}, on a un diagramme commutatif de suites exactes: 
$$\xymatrix{H^1(Y,S)\ar[r]^-{\chi}&\Hom_{D^+(k)}(S^*,KD'(Y))\ar[d]^{\psi_3}\ar[r]&H^2(k,S)\\
H^1(Y_{\sigma},S_{\sigma})\ar[r]\ar[d]&\Hom_{D^+(k)}(S_{\sigma}^*,KD'(Y_{\sigma}))\ar[r]\ar[d]^{\psi_4}&H^2(k,S_{\sigma})\ar[d]\\
H^1(X',S_{\sigma})\ar[r]^-{\chi_S}\ar[d]&\Hom_{D^+(k)}(S_{\sigma}^*,KD'(X'))\ar[r]^-{\partial_S}\ar[d]^{-\circ \phi_{\sigma}^*}&H^2(k,S_{\sigma})\ar[d]\\
H^1(X',T_{\sigma})\ar[r]^-{\chi_T}&\Hom_{D^+(k)}(T_{\sigma}^*,KD'(X'))\ar[r]^-{\partial_T}&H^2(k,T),
}$$
o\`u $\psi_3$ est induit par le tordu de $\sigma$ (voir \cite[Lem. 7.2]{CDX}) et $\psi_4$ est induit par $X'\sbt Y_{\sigma}$ (et donc $KD'(Y_{\sigma})\to KD'(X')$).
Appliquons \cite[Lem. 7.3]{CDX} au $L'$-torseur $\psi_{2,*}([V])$ sur $X$. 
On obtient que $\sigma\in \Im (\psi_{1,*})=\Im(\psi_*)$ si et seulement si 
$$\partial_S((\psi_4\circ \psi_3\circ \chi)([V]_Y))= \partial_T((\psi_4\circ \psi_3\circ \chi)([V]_Y)\circ \phi_{\sigma}^*)=0,$$ o\`u $[V]_Y$ est le $S$-torseur $V\to Y$.
D'apr\`es la proposition \ref{sec2prop1}, le type $\chi_S$ est surjectif, d'o\`u le r\'esultat.
\end{proof}

\section{D\'emonstration}

Dans toute cette section,  $k$ est un corps de nombres. 
Sauf  mention explicite du contraire,  une vari\'et\'e est une $k$-vari\'et\'e.

Les d\'efinitions $X(\RA_k)^{\et,\Br}, X(\RA_k)^{G-\et, \Br_G},  X(\RA_k)^{G-\et, \Br'_G}$ peuvent \^etre g\'en\'eralis\'es \`a la fa\c{c}on ci-dessous (\ref{def1e2}).

Soit $\mathbf{AB}$ la cat\'egorie des groupes ab\'eliens. 
Soit $\mathbf{GX}$ la cat\'egorie  des couples $(G,X)$ avec $G$ un groupe alg\'ebrique connexe et $X$ une $G$-vari\'et\'e lisse,
 et un morphisme $(H,Y)\to (G,X)$ dans $\mathbf{GX}$ est un couple $(\psi,f)$ avec $\psi: H\to G$ un homomorphisme et $f: Y\to X$ un $H$-morphisme,
  o\`u l'action de $H$ sur $X$ est induite par $\psi$.
  
  En fait, pour tout torseur $f: Y\to X$ sous un $k$-groupe fini, il existe un groupe alg\'ebrique connexe $H$ et un homomorphisme fini surjectif $\psi: H\to G$
   tels que $(\psi,f)\in \mathbf{GX}$ et $\psi$ soit minimal pour cette propri\'et\'e (cf. \cite[Prop. 3.8]{C5}).
  Ce groupe $H$ s'appelle \emph{le groupe minimal compatible avec le torseur $f$} (\cite[D\'ef. 3.9]{C5}).
  
  Soit  $B(-,-): \mathbf{GX}\to \mathbf{AB}$ un foncteur contravariant qui associe au couple $(G,X)$ un sous-groupe $B(G,X)\sbt \Br(X)$.
  On d\'efinit
  \begin{equation}\label{def1e2}
  X(\RA_k)^{G-\et,B(G,-)}:= \bigcap_{\stackrel{f: Y\xrightarrow{F}X\ G-\text{compatible} ,}{ F\ \text{fini}}} \bigcup_{\sigma\in H^1(k,F)}f_{\sigma}(Y_{\sigma}(\RA_k)^{B(G,Y_{\sigma})}) 
  \end{equation}
  
  On fixe un objet $(G,X)\in \mathbf{GX}$. 
Soit $\mathbf{GX}_X$ l'ensemble des objets $(H,Y)\in \mathbf{GX}$ tels qu'il existe un morphisme $(\psi,f): (H,Y)\to (G,X)$ dans $\mathbf{GX}$ avec $\psi, f$ finis.
En fait, la d\'emonstration de \cite[Thm. 1.4]{C5} montre le th\'eor\`eme ci-dessous.

\begin{thm}(\cite[Rem. 6.7]{C5})\label{thmmainpfrem}
Soit $B(-,-): \mathbf{GX}\to \mathbf{AB}$ un foncteur contravariant comme ci-dessus.
Supposons que, pour tout entier $n \geq 2$ et tout objet $(H,Y)$ dans $\mathbf{GX}_X$ avec $Y$ g\'eom\'e\-triquement int\`egre,
 on a:

(i) si $Y(\RA_k)^{B(H,Y)}\neq\emptyset$, alors il existe un torseur universel de $n$-torsion pour $Y$;

 (ii)  pour tout  \'el\'ement de $n$-torsion $\alpha\in \Br(Y)$ et tout torseur universel de $n$-torsion $f: \CT_Y\to Y$ (s'il existe) sous le groupe $S_Y$, 
 on a $f^*(\alpha)\in B(H',\CT_Y)$, o\`u $H'$ est le groupe minimal compatible avec le torseur $f$;
 
 (iii) pour tout $(H',Y')\in \mathbf{GX}_X$, si  $Y'=\sqcup_{i\in I} Y_i$ avec $Y_i$ g\'eom\'etriquement int\`egre pour tout $i\in I$, on a 
 $$Y'(\RA_k^{nc})^{B(H',Y')}=\coprod_{i\in I}  Y_i(\RA_k^{nc})^{B(H',Y_i)} \ \ \ \text{et} \ \ \  
Y'(\RA_k^{nc})^{H'-\et,B(H',-)}=\coprod_{i\in I}  Y_i(\RA_k^{nc})^{H'-\et, B(H',-)} ;$$
 
 (iv) pour une extension centrale de $k$-groupes finis $1\to S \to L \xrightarrow{\psi} F\to 1$,
un $\sigma\in H^1(k,F)$, un $L$-torseur $ V\to Y$ avec $Z:=V/S \to Y$ le $F$-torseur induit par $\psi$,
alors $Z_{\sigma}(\RA_k)^{B(H',Z_{\sigma})}\neq\emptyset$ implique $\sigma\in \Im(H^1(k,L)\to H^1(k,F))$, where $H'$ est le groupe minimal compatible avec le torseur $Z_{\sigma}\to Y$;
 
 (v) pour un morphisme $(\psi, f): (H',Y')\to (H,Y)$ dans $ \mathbf{GX}_X$ tel que $\psi$ soit fini surjectif de noyau $S$ central, $Y'$ soit g\'eom\'e\-triquement int\`egre et $f$ soit un $S$-torseur, o\`u l'action de $S$ est induite par l'action de $H'$,
on a, pour tout $\sigma\in H^1(k,S)$, le tordu $Y'_{\sigma}$ est une $H'$-vari\'et\'e et:
$$Y(\RA_k)^{B(H,Y)}=\cup_{\sigma\in H^1(k,S)}f_{\sigma}(Y'_{\sigma}(\RA_k)^{B(H',Y'_{\sigma})}).$$
 
Alors on a $X(\RA_k)^{\et,\Br}=X(\RA_k)^{G-\et,B(G,-)}$.
\end{thm}

Rappelons que, soit $S_X$ un $k$-groupe fini commutatif avec $S_X^*\cong H^1(X_{\bk},\mu_n)$,
  un \emph{torseur universel de $n$-torsion} pour $X$ est un $S_X$-torseur $\CT_X$ sur $X$
 tel que $\chi([\CT_X])$ dans (\ref{sec2e1}) soit l'homomorphisme canonique $H^1(X_{\bk},\mu_n)\to KD'(X)$ (cf. \cite[D\'ef. 2.1]{C5}).

\begin{lem}\label{thmmainpflem}
Soit $B(-,-): \mathbf{GX}\to \mathbf{AB}$ un foncteur contravariant comme ci-dessus.
Si $\Br_{2/3}(Y)\sbt B(H,Y)$ pour tout $(H,Y)\in \mathbf{GX}_X$, alors les hypoth\`eses (i) et (iv) du th\'eor\`eme \ref{thmmainpfrem} valent.
\end{lem}

\begin{proof}
La proposition \ref{prop21.3} implique (iv). 

La proposition \ref{sec2prop1} implique, si $Y(\RA_k)^{\Br_{2/3}(Y)}\supset Y(\RA_k)^{B(H,Y)}\neq\emptyset$, alors $\chi$ dans (\ref{sec2e1}) est surjectif, d'o\`u on a (i). 
\end{proof}

\begin{proof}[D\'emonstration du th\'eor\`eme \ref{Thm1}]
Par d\'efinition, il suffit de montrer que $$X(\RA_k)^{G-\et, \Br'_G}=X(\RA_k)^{\et, \Br} .$$
D'apr\`es le th\'eor\`eme \ref{thmmainpfrem}, il suffit de v\'erifier les hypoth\`eses (i)-(v) du th\'eor\`eme \ref{thmmainpfrem}
 pour $B(-,-):=\Br'_{(-)}(-)$.
 
Le lemme \ref{thmmainpflem} et la proposition \ref{prop21.2} (1) impliquent (i) et (iv).
Le lemme \ref{lem4.22.1} implique (iii).
La proposition \ref{descprop4.1} implique (v).
\`A la fin, l'hypoth\`ese (ii) d\'ecoule du m\^eme argument que \cite[Prop. 3.12]{C5}.
On le rappelle ci-dessous.

Notons $\psi: H'\to H$ l'homomorphisme, $\rho: H\times Y\to Y$, $\rho_t: H'\times \CT_Y\to \CT_Y$ les actions 
et $p_1:H\times Y\to H$, $p_2:H\times Y\to Y$, $p_{1,t}: H'\times \CT_Y\to H'$, $p_{2,t}: H'\times \CT_Y\to \CT_Y$ les projections.
Soit $\CT_H$ un torseur universel de $n$-torsion pour $H$ sous le groupe $S_H$.

La formule de K\"unneth de degr\'e 2 (\cite{SZ}, cf. \cite[Cor. 2.7]{C5}) donne un homomorphisme surjectif
$$H^2(H,\mu_n)\oplus H^2(Y,\mu_n)\oplus \Hom_k(S_H,S_Y^*)\xrightarrow{(p_1^*,p_2^*,\varepsilon) }H^2(H\times Y,\mu_n),$$
o\`u $\varepsilon (\phi)=\phi_*([\CT_H])\cup [\CT_Y] $.
On obtient :
pour tout $\alpha_1\in H^2(Y, \mu_n)$, il existe un $\phi \in \Hom(S_H,S_Y^*)$ et un $\beta\in H^2(H,\mu_n)$ 
tels que $(\rho^*-p_2^*)(\alpha_1)=\varepsilon (\phi)+p_1^*(\beta)$.

Puisque $f^*([\CT_Y])=0\in H^1(\CT_Y,S_Y)$,  on a
$$(\psi\times f)^*(\varepsilon (\phi))=(\psi\times f)^* (\phi_*([\CT_H])\cup [\CT_Y] )=\phi_*(\psi^*([\CT_H]))\cup f^*([\CT_Y])=0.$$
Alors $(\rho_t^*-p_{2,t}^*)(f^*(\alpha_1) )= (\psi\times f)^*((\rho^*-p_2^*)(\alpha_1))=(\psi\times f)^*(p_1^*(\beta))=p_{1,t}^*(\psi^*(\beta))$.
D'apr\`es la suite exacte de Kummer, $(\rho_t^*-p_{2,t}^*)(f^*(\alpha))\sbt p_{1,t}^*\Im(H^2(H',\mu_n)\to \Br(H')) $, d'o\`u le r\'esultat.
\end{proof}

\begin{cor}\label{cormainthm}
Soit $X$ un $G$-espace homog\`ene \`a stabilisateur g\'eom\'etrique connexe.
Alors $X(\RA_k)^{\Br'_G(X)}=X(\RA_k)^{\et, \Br} $.
\end{cor}

\begin{proof}
D'apr\`es  \cite[Cor. 3.5 (4)]{C5}, tout torseur  $G$-compatible  sous un $k$-groupe fini provient de $k$.
Donc on a $X(\RA_k)^{G-\et, \Br'_G}=X(\RA_k)^{\Br'_G(X)}$.
Le r\'esultat d\'ecoule du th\'eor\`eme \ref{Thm1}.
\end{proof}

\bibliographystyle{alpha}
\end{document}